\documentclass{amsart}
\usepackage[utf8]{inputenc}
\usepackage{amsmath}
\usepackage{amssymb}
\usepackage{mathrsfs}
\usepackage{graphicx}

\usepackage[colorlinks]{hyperref}
\usepackage[svgnames, dvipsnames, usenames]{xcolor}
\hypersetup{urlcolor = RedViolet, linkcolor = RoyalBlue, citecolor = ForestGreen}

\usepackage{tikz}
\usepackage{partmac}
\usepackage{graphmor}

\def\ppen{\penalty 300 }
\let\col=\colon
\def\colon{\col\ppen}

\theoremstyle{plain} 
\newtheorem{thm}{Theorem}[section]
\newtheorem{prop}[thm]{Proposition}
\newtheorem{lem}[thm]{Lemma}
\newtheorem{cor}[thm]{Corollary}

\theoremstyle{definition}
\newtheorem{defn}[thm]{Definition}
\newtheorem{rem}[thm]{Remark}
\newtheorem{ex}[thm]{Example}

\newtheorem{nota}[thm]{Notation}

\numberwithin{equation}{section}

\theoremstyle{plain}

\renewcommand{\theta}{\vartheta}
\renewcommand{\phi}{\varphi}
\renewcommand{\epsilon}{\varepsilon}
\renewcommand{\subset}{\subseteq}

\newcommand{\N}{\mathbb N}
\newcommand{\Z}{\mathbb Z}
\newcommand{\T}{\mathbb T}

\newcommand{\R}{\mathbb R}
\newcommand{\C}{\mathbb C}
\newcommand{\Cl}{\mathrm{Cl}}
\newcommand{\staralg}{\mathop{\rm\ast\mathchar `\-alg}}
\DeclareMathOperator{\Mor}{Mor}
\DeclareMathOperator{\id}{id}

\DeclareMathOperator{\Aut}{Aut}
\DeclareMathOperator{\Irr}{Irr}
\DeclareMathOperator{\spanlin}{span}
\DeclareMathOperator{\tr}{tr}
\DeclareMathOperator{\Tr}{Tr}
\DeclareMathOperator{\Rep}{Rep}
\DeclareMathOperator{\rank}{rank}
\DeclareMathOperator{\End}{End}

\newcommand{\RCat}{\mathfrak{C}}

\newcommand{\Lin}{\mathscr{L}}

\newcommand{\Cay}{\mathrm{Cay}}
\newcommand{\F}{\mathcal{F}}

\newcommand{\Olg}{\mathscr{O}}

\newcommand{\im}{\mathrm{i}}
\newcommand{\glin}{\mathfrak{gl}}
\newcommand{\slin}{\mathfrak{sl}}
\newcommand{\un}{\mathfrak{u}}
\newcommand{\su}{\mathfrak{su}}
\newcommand{\Gr}{\mathfrak{G}}

\def\makebcirc#1#2{\ooalign{$#1\bigcirc\mathsurround=0pt$\cr \hfil$#1#2\mathsurround=0pt$\hfil}}
\newcommand{\OT}{\mathbin{\mathpalette\makebcirc{\perp}}}

\newcommand{\mergepart}{\Partition{
\Pblock 1to0.5:1,2
\Psingletons 0to0.5:1.5
}}

\makeatletter
\def\widebreve{\mathpalette\wide@breve}
\def\wide@breve#1#2{\sbox\z@{$#1#2$}%
     \mathop{\vbox{\m@th\ialign{##\crcr
\kern0.08em\brevefill#1{0.8\wd\z@}\crcr\noalign{\nointerlineskip}%
                    $\hss#1#2\hss$\crcr}}}\limits}
\def\brevefill#1#2{$\m@th\sbox\tw@{$#1($}%
  \hss\resizebox{#2}{\wd\tw@}{\rotatebox[origin=c]{90}{\upshape(}}\hss$}
\makeatletter

\begin{document}
\title{Some examples of quantum graphs}
\author{Daniel Gromada}
\address{Czech Technical University in Prague, Faculty of Electrical Engineering, Department of Mathematics, Technická 2, 166 27 Praha 6, Czechia}
\email{gromadan@fel.cvut.cz}
\thanks{I would like to thank Simon Schmidt, Christian Voigt, and Moritz Weber for valuable discussions about quantum graphs. I thank to Julien Bichon for discussing monoidal equivalences related to $SO_n$.}
\thanks{This work was supported by the project OPVVV CAAS CZ.02.1.01/0.0/0.0/16\_019/0000778}
\date{\today}

\begin{abstract}
We summarize different approaches to the theory of quantum graphs and provide several ways to construct concrete examples. First, we classify all undirected quantum graphs on the quantum space $M_2$. Secondly, we apply the theory of 2-cocycle deformations to Cayley graphs of abelian groups. This defines a twisting procedure that produces a quantum graph, which is quantum isomorphic to the original one. For instance, we define the anticommutative hypercube graphs. Thirdly, we construct an example of a quantum graph, which is not quantum isomorphic to any classical graph.
\end{abstract}

\maketitle
\section*{Introduction}

Quantum graphs have been introduced very recently and hence form a very young and perspective field to study. In order to briefly review their history, we should go back to the work of Wang \cite{Wan98} and Banica \cite{Ban99,Ban02}, who interpreted finite-dimensional C*-algebras as finite quantum spaces and studied their quantum symmetries. Independently on these results, finite quantum sets were defined in terms of Frobenius algebras in \cite{CPV13}. Quantum graphs (or quantum relations) were first defined by Weaver in \cite{Wea12,Wea21}. He did not build on the earlier results we just mentioned, actually he did not assume any finiteness at all working over general von Neumann algebras. Finite quantum graphs were then reformulated to the framework of Frobenius algebras in \cite{MRV18} and then further studied from the categorical perspective in \cite{MRV19}. To close the circle, quantum graphs were also studied in the C*-algebraical framework in \cite{BCE+20} and found some connections to the theory of operator algebras in \cite{BEVW20}.

The character of this article is to large extent expository. We aim to summarize different approaches to quantum graphs and present some results from new viewpoints. Most importantly, however, this work contains several original concrete examples of quantum graphs and presents techniques of how to construct such examples. The classical graph theory tends to be very graphical and builds on many concrete examples of interesting graphs. In contrast, the literature on quantum graphs is missing such examples. We hope to convince the reader that constructing quantum graphs is almost as easy as constructing the classical ones.

The article can be divided into two parts. In the first part (Sections \ref{sec.qg}--\ref{sec.concrete}) we study quantum graphs as such. We bring three equivalent definitions of quantum graphs over a quantum space $X$ being based on
\begin{enumerate}
\item \emph{edge projection}, i.e.\ a projection $\tilde A\in C(X)\otimes C(X)^{\rm op}$ \cite[Section~VII]{MRV18},
\item \emph{adjacency matrix}, i.e.\ an operator $A\colon l^2(X)\to l^2(X)$ \cite[Def.~V.1]{MRV18},
\item an operator space $V$, such that $C(X)'VC(X)'\subset V$ \cite{Wea12}.
\end{enumerate}

In Section~\ref{sec.constructions}, we define some basic quantum graph constructions such as quantum graph isomorphism, subgraphs, quotient graphs, multiple edges, weighted graphs. In Section~\ref{secc.smallclass}, we bring the first set of examples -- we classify all quantum graphs over the quantum space $M_2$ given by $2\times 2$ matrices. This result was actually very recently independently obtained by Junichiro Matsuda \cite{Mat21} (his result is actually more general since he considers also quantum graphs with non-tracial $\delta$-form).

In the second part of the article, we study quantum graphs from the viewpoint of their quantum symmetries and related representation categories. In Section~\ref{sec.qsym}, we recall the definition of the quantum automorphism group and quantum isomorphisms of quantum graphs.

The main result of the whole article is presented in Section~\ref{sec.twist}, where we interpret the theory of comodule algebra twists (see e.g.\ \cite[Section~7.5]{Mon93}) in terms of quantum graphs. Given a finite abelian group $\Gamma$, we consider the twist of $C(\Gamma)\simeq\C\Gamma$ taken as a comodule algebra. We show that this construction extends to Cayley graphs of $\Gamma$. This allows to construct quantum analogues of such Cayley graphs, which are then quantum isomorphic to the original ones. This construction can be actually viewed as a special case of a more general 2-cocycle twisting procedure formulated in \cite{MRV19}. Nevertheless, the case of Cayley graphs is much more straightforward.

In Section~\ref{sec.rook}, we illustrate the twisting on the series of groups $\Gamma=\Z_n\times\Z_n$. The classical space of $n^2$ points represented by the algebra $C(\Gamma)$ is twisted to the quantum space $M_n$ defined through the matrix algebra $M_n(\C)$. This way, we can for instance define the quantum rook's graph. Besides that, this gives an alternative proof of a well-known fact that the classical space of $n^2$ points is quantum isomorphic with the quantum space $M_n$, which consequently implies that all quantum spaces of fixed size (the dimension of the algebra) are quantum isomorphic. (See also \cite{Ban99,Ban02,RV10,GProj}.)

In Section~\ref{sec.cube}, we study the case $\Gamma=\Z_2^n$, i.e.\ the cube-like graphs. Here, the comodule algebra twist gives us the Clifford algebra \cite{AM02}, which then forms the underlying space for \emph{quantum cube-like graphs}.

Finally, in Section~\ref{sec.smallnoniso}, we get back to the examples of quantum graphs over $M_n$ with small $n$. We show that all the simple quantum graphs over $M_2$ that were classified in Section~\ref{secc.smallclass} are examples of quantum Cayley graphs and hence are quantum isomorphic to their classical counterparts. In contrast, we construct an example of a simple graph on $M_3$, which is \emph{not} quantum isomorphic to a classical graph.

\subsection*{Notation}
To improve the clarity of the article, we are going to use two different symbols for involution -- the asterisk $*$ and the dagger $\dag$. We will usually use $*$ to denote the $*$ operation in a C*-algebra or a $*$-algebra. In contrast, we will use $\dag$ to denote taking an adjoint of concrete linear maps between Hilbert spaces.

In particular, note the following distinction: Consider a C*-algebra $C(X)$ equipped with a trace and the associated GNS Hilbert space $l^2(X)$. An element $x\in l^2(X)$ can be interpreted as an element of $C(X)$ in which case we denote by $x^*$ its adjoint with respect to the $*$-operation in $C(X)$. Alternatively, we can interpret $x$ as a linear operator $\C\to l^2(X)$, in which case we denote by $x^\dag\colon l^2(X)\to\C$ its Riesz dual with respect to the inner product in $l^2(X)$.

We are also going to speak about representation categories, which are concrete monoidal involutive categories, so we are going to denote the involution by $\dag$ and refer to them as $\dag$-categories.

Sometimes, we are going to illustrate some equations using string diagrams. We decided to be consistent with \cite{Vic11,MRV18,MRV19} and inconsistent with the literature on partition categories (in particular, with \cite{GroAbSym}), so all diagrams are now to be read from bottom to top.


We will often work with operators between some tensor powers of some vector spaces. Therefore, we adopt the ``physics notation'' with upper and lower indices for entries of these ``tensors''. That is, given $T\colon V^{\otimes k}\to V^{\otimes l}$ for some $V=\C^N$, we denote
$$T(e_{i_1}\otimes\cdots\otimes e_{i_k})=\sum_{j_1,\dots,j_l=1}^NT^{j_1\cdots j_l}_{i_1\cdots i_k}(e_{j_1}\otimes\cdots\otimes e_{j_l})$$
We will sometimes shorten the notation and write $T_\mathbf{i}^\mathbf{j}$ using multiindices $\mathbf{i}=(i_1,\dots,i_k)$, $\mathbf{j}=(j_1,\dots,j_l)$.

If not mentioned otherwise, all ($*$-)algebras and all ($*$-)homomorphisms are always assumed to be unital.


\section{Quantum graphs}
\label{sec.qg}

\subsection{Quantum sets}

Let $A$ be a finite-dimensional C*-algebra and let $\psi$ be a faithful state on $A$. By the GNS construction, this defines an inner product $\langle a,b\rangle=\psi(a^*b)$. Likewise, we can define a non-degenerate bilinear form $(a,b)=\psi(ab)$. This gives $A$ the structure of a \emph{Frobenius algebra}.

Let us denote the multiplication on $A$ by $m\colon A\otimes A\to A$. Now we can compute its adjoint $m^\dag\colon A\to A\otimes A$ with respect to the above defined inner product. We say that the state $\psi$ defining the inner product is a \emph{$\delta$-form} if $mm^\dag=\delta^2\,\id$. 

In the theory of Frobenius algebras, one typically requires that $mm^\dag=\id$ without the factor $\delta^2$. This can be achieved by rescaling the definition of the inner product. That is, we may rather define $\langle a,b\rangle=\delta^2\psi(a^*b)$ and $(a,b)=\delta^2\psi(ab)$. Such a Frobenius algebra is called \emph{special}. From now on, we will only use this scaled version of the inner product and the bilinear form. If necessary, we will denote this bilinear form also by $R^\dag\colon A\otimes A\to\C$.

Denote by $\eta$ the unit of the algebra $A$ taken as a map $\eta\colon\C\to A$. By definition of the inner product, we have $\eta^\dag=\delta^2\,\psi$ (recall that we work with the scaled version of the inner product; for the standard GNS inner product, we would have $\eta^\dag=\psi$). In the following text, we will typically work only with the functional $\eta^\dag$ and call it the \emph{counit} and we will try to avoid using the state $\psi$ to prevent confusion about the scaling factors. Once again, recall that $\langle a,b\rangle=\eta^\dag(a^*b)$. (This counit $\eta^\dag$ should not be confused with the quantum group counit $\epsilon$.)

If the state $\psi$ is tracial, then the bilinear form $(a,b)=\delta^2\psi(ab)=\eta^\dag(ab)$ is symmetric and the corresponding Frobenius algebra is also called \emph{symmetric}. Any finite-dimensional C*-algebra is of the form $A=\bigoplus_i M_{n_i}(\C)$ and possesses a unique tracial $\delta$-form \cite[Proposition~2.1]{Ban99}. This gives $A$ a unique structure of a symmetric Frobenius algebra. The counit is in this case of the form
$$\eta^\dag(x)=\sum_i n_i\Tr(x_i).$$
See \cite{Vic11} for more information about Frobenius algebras and their relation with C*-algebras.

This serves as the basis for a definition of a quantum space.

\begin{defn}
A \emph{finite quantum set} $X$ of size $N$ is an $N$-dimensional C*-algebra equipped with the unique structure of a symmetric Frobenius algebra as described above. We denote by $C(X)$ the underlying C*-algebra and by $l^2(X)$ the associated Hilbert space (with the rescaled inner product).
\end{defn}

The history of this concept goes back to Wang \cite{Wan98}, who studied the quantum symmetries of the classical set of $N$ points as well as the quantum set formed by the algebra of matrices $M_n(\C)$. His ideas were further developed by Banica \cite{Ban99,Ban02}, who also introduced the notion of a \emph{$\delta$-form}. Quantum sets in terms of Frobenius algebras were defined in \cite{MRV18} based on a certain ``finite Gelfand duality'' result from \cite{CPV13}.

Note that some authors (e.g.~\cite{BCE+20}) drop the condition that the Frobenius algebra is symmetric (the state is a trace). However, all quantum spaces and graphs constructed here will be defined in the symmetric (tracial) setting.

In order to make everything clear, let us state two obvious examples.

\begin{ex}[Classical finite set]
We will denote by $X_N$ the classical finite set of $N$ elements. This is a special case of a quantum set: We take the algebra $C(X_N)$ of functions on $X_N$ as our C*-algebra. The tracial state $\psi$ is the normalized counting measure $\psi(f)=\frac{1}{N}\sum_{x\in X_N}f(x)$. This is a $\delta$-form with $\delta^2=N$. Rescaling it, we get the ordinary counting measure $\eta^\dag(f)=\sum_{x\in X_N}f(x)$ corresponding to the standard inner product on $l^2(X)$ given by $\langle f,g\rangle=\eta^\dag(f^*g)=\sum_{x\in X_N}\overline{f(x)}g(x)$.
\end{ex}

\begin{ex}[Matrix algebra]
\label{ex.Mn}
Another obvious example is the matrix algebra $A=M_n(\C)$. We will denote by $M_n$ the corresponding finite quantum set with $C(M_n)=M_n(\C)$. Here, we take the normalized trace as a state $\psi=\tr$, which is a $\delta$-form with $\delta^2=n^2$, so the counit is given by $\eta^\dag=n^2\tr=n\Tr$, where $\Tr$ is the standard trace. Denoting by $(e_{ij})$ the standard basis of matrices with 1 at position $(i,j)$ and 0 otherwise, we have that $\langle e_{ij},e_{kl}\rangle=n\delta_{ik}\delta_{jl}$, so it forms an orthogonal basis with every element having norm $\sqrt n$. The multiplication expressed in this basis has entries $m^{ab}_{cdef}=\delta_{ac}\delta_{bf}\delta_{de}=\frac{1}{n}[m^\dag]^{cdef}_{ab}$.
\end{ex}

\begin{lem}\label{L.RF}
Let $(e_i)$ be an orthonormal basis of $l^2(X)$. Denote by $(F_i^j)$ the matrix of the $*$-operation on $C(X)$, that is, $e_i^*=\sum_jF_i^je_j$. Then $F_i^j=R^{ij}$. In addition, $F^{-1}=\bar F$.
\end{lem}
\begin{proof}
For the first formula, we compute directly
$$R^{ij}=\overline{[R^\dag]_{ij}}=\overline{\langle e_i^*,e_j\rangle}=\sum_k F_i^k\delta_{kj}=F_i^j.$$

For the second one, we apply the $*$-operation twice. Since $*$ is involutive, we have $e_i=e_i^{**}=\sum_j\bar F_i^je_j^*=\sum_{jk}\bar F_i^jF_j^ke_k$, so $F\bar F=\id$.
\end{proof}

\subsection{Diagrams}

It is often convenient to formulate abstract equations in Frobenius algebras using certain string diagrams. We introduce the following notation

\begin{align*}
    m&=:\Gfork   &   m^\dag&=:\Gmerge\\
 \eta&=:\Gupsing &\eta^\dag&=:\Gsing\\
    R&=m^\dag\eta=
\Graph{
\draw (1.5,0) node{} -- (1.5,0.5);
\draw (1,1) -- (1,0.7) -- (1.5,0.5) node {} -- (2,0.7) -- (2,1);
}
=:
\Guppair
& R^\dag&=\eta^\dag m=
\Graph{
\draw (1.5,1) node{} -- (1.5,0.5);
\draw (1,0) -- (1,0.3) -- (1.5,0.5) node {} -- (2,0.3) -- (2,0);
}=:\Gpair
\end{align*}

As an illustration, we can express the property of the algebra identity $1_C(X)a=a=a1_{C(X)}$. In terms of the mappings $\eta$ and $m$, it can be written as $(\eta\otimes\id)m=\id=(\id\otimes\eta)m$; consequently, we have an analogous condition for the counit and comultiplication $m^\dag(\eta^\dag\otimes\id)=\id=m^\dag(\id\otimes\eta^\dag)$. In terms of diagrams, those conditions read
\begin{equation}\label{eq.unity}
\Graph{
\draw (1.5,0.5) -- (2,0.3) -- (2,0);
\draw (1.5,1) -- (1.5,0.5) node {} -- (1,0.3) -- (1,0) node {};
}=
\Graph{
\draw (1,1)--(1,0);
}
=
\Graph{
\draw (1.5,0.5) -- (2,0.3) -- (2,0)node{};
\draw (1.5,1) -- (1.5,0.5) node {} -- (1,0.3) -- (1,0);
},\qquad
\Graph{
\draw (1.5,0.5) -- (2,0.7) -- (2,1);
\draw (1.5,0) -- (1.5,0.5) node{} -- (1,0.7) -- (1,1) node{};
}=
\Graph{
\draw (1,1)--(1,0);
}
=
\Graph{
\draw (1.5,0.5) -- (2,0.7) -- (2,1) node{};
\draw (1.5,0) -- (1.5,0.5) node{} -- (1,0.7) -- (1,1);
},\qquad
\end{equation}

As another example, let us mention the formula known as the \emph{Frobenius law}. Note that Frobenius algebras are sometimes defined in a more abstract way, where the Frobenius law is an axiom.

\begin{prop}\label{P.Flaw}
Let $A$ be a symmetric Frobenius algebra. Then the following holds:
$$(m\otimes\id)(\id\otimes m^\dag)=m^\dag m=(\id\otimes m)(m^\dag\otimes\id).$$
In the language of diagrams:
\begin{equation}\label{eq.Flaw}
\Graph{
\draw (1,0) -- (1,0.8) -- (1.5,1);
\draw (2.5,1.5) -- (2.5,0.7) -- (2,0.5);
\draw (1.5,1.5) -- (1.5,1) node{} -- (2,0.5) node{} -- (2,0);
}
=
\Graph{
\draw (1.5,0.5) -- (1.5,1);
\draw (1,0) -- (1,0.3) -- (1.5,0.5) node {} -- (2,0.3) -- (2,0);
\draw (1,1.5) -- (1,1.2) -- (1.5,1) node {} -- (2,1.2) -- (2,1.5);
}
=
\Graph{
\draw (2.5,0) -- (2.5,0.8) -- (2,1);
\draw (1,1.5) -- (1,0.7) -- (1.5,0.5);
\draw (2,1.5) -- (2,1) node{} -- (1.5,0.5) node{} -- (1.5,0);
}.
\end{equation}
\end{prop}

See e.g.~\cite[Lemma~3.17]{Vic11} for a proof.

\begin{cor}\label{C.Flaw}
We have
\begin{enumerate}
\item $(R^\dag\otimes\id)(\id\otimes R)=\id=(\id\otimes R^\dag)(R\otimes\id)$, i.e.
$
\Graph{\draw (1,0) -- (1,0.8) -- (1.5,1) node{} -- (2,0.8) -- (2,0.2) -- (2.5,0) node{} -- (3,0.2) -- (3,1);}
=
\Graph{\draw (1,1)--(1,0);}
=
\Graph{\draw (3,0) -- (3,0.8) -- (2.5,1) node{} -- (2,0.8) -- (2,0.2) -- (1.5,0) node{} -- (1,0.2) -- (1,1);},
$
\item[(2a)] $(\id\otimes m)(R\otimes\id)=m^\dag=(m\otimes\id)(\id\otimes R)$, i.e.\ 
$
\Graph{\draw(1.5,1)--(1.5,1.5);\draw (1,0) -- (1,0.8) -- (1.5,1) node{} -- (2,0.8) -- (2,0.2) -- (2.5,0) node{} -- (3,0.2) -- (3,1.5);}
=
\Gmerge
=
\Graph{\draw(2.5,1)--(2.5,1.5);\draw (3,0) -- (3,0.8) -- (2.5,1) node{} -- (2,0.8) -- (2,0.2) -- (1.5,0) node{} -- (1,0.2) -- (1,1.5);},
$
\item[(2b)] $(\id\otimes m^\dag)(R^\dag\otimes\id)=m=(m^\dag\otimes\id)(\id\otimes R^\dag)$, i.e.\ 
$
\Graph{\draw(1.5,0)--(1.5,-0.5);\draw (3,-0.5) -- (3,0.8) -- (2.5,1) node{} -- (2,0.8) -- (2,0.2) -- (1.5,0) node{} -- (1,0.2) -- (1,1);},
=
\Gfork
=
\Graph{\draw(2.5,0)--(2.5,-0.5);\draw (1,-0.5) -- (1,0.8) -- (1.5,1) node{} -- (2,0.8) -- (2,0.2) -- (2.5,0) node{} -- (3,0.2) -- (3,1);}
$
\end{enumerate}
\end{cor}
\begin{proof}
Item (1):
$$
\Graph{\draw (1,0) -- (1,0.8) -- (1.5,1) node{} -- (2,0.8) -- (2,0.2) -- (2.5,0) node{} -- (3,0.2) -- (3,1);}
=
\Graph{\draw (1.5,1)--(1.5,1.5) node{};\draw (2.5,0)--(2.5,-0.5) node{};\draw (1,0) -- (1,0.8) -- (1.5,1) node{} -- (2,0.8) -- (2,0.2) -- (2.5,0) node{} -- (3,0.2) -- (3,1);}
\buildrel\eqref{eq.Flaw}\over=
\Graph{
\draw (1.5,0.2) -- (1.5,0.8);
\draw (1,-0.5) -- (1,0) -- (1.5,0.2) node {} -- (2,0) -- (2,-0.5) node{};
\draw (1,1.5) node{} -- (1,1) -- (1.5,0.8) node {} -- (2,1) -- (2,1.5);
}
\buildrel\eqref{eq.unity}\over=
\Graph{\draw(1,1.5) -- (1,-0.5);}
\buildrel\eqref{eq.unity}\over=
\Graph{
\draw (1.5,0.2) -- (1.5,0.8);
\draw (1,-0.5) node{} -- (1,0) -- (1.5,0.2) node {} -- (2,0) -- (2,-0.5);
\draw (1,1.5) -- (1,1) -- (1.5,0.8) node {} -- (2,1) -- (2,1.5) node{};
}
\buildrel\eqref{eq.Flaw}\over=
\Graph{\draw (2.5,1)--(2.5,1.5) node{};\draw (1.5,0)--(1.5,-0.5) node{};\draw (3,0) -- (3,0.8) -- (2.5,1) node{} -- (2,0.8) -- (2,0.2) -- (1.5,0) node{} -- (1,0.2) -- (1,1);},
=
\Graph{\draw (3,0) -- (3,0.8) -- (2.5,1) node{} -- (2,0.8) -- (2,0.2) -- (1.5,0) node{} -- (1,0.2) -- (1,1);}
$$
Item (2a):
$$
\Graph{\draw(1.5,1)--(1.5,1.5);\draw (1,0) -- (1,0.8) -- (1.5,1) node{} -- (2,0.8) -- (2,0.2) -- (2.5,0) node{} -- (3,0.2) -- (3,1.5);}
=
\Graph{\draw (1.5,1)--(1.5,1.5);\draw (2.5,0)--(2.5,-0.5) node{};\draw (1,0) -- (1,0.8) -- (1.5,1) node{} -- (2,0.8) -- (2,0.2) -- (2.5,0) node{} -- (3,0.2) -- (3,1.5);}
\buildrel\eqref{eq.Flaw}\over=
\Graph{
\draw (1.5,0.2) -- (1.5,0.8);
\draw (1,-0.5) -- (1,0) -- (1.5,0.2) node {} -- (2,0) -- (2,-0.5) node{};
\draw (1,1.5) -- (1,1) -- (1.5,0.8) node {} -- (2,1) -- (2,1.5);
}
\buildrel\eqref{eq.unity}\over=
\Gmerge
\buildrel\eqref{eq.unity}\over=
\Graph{
\draw (1.5,0.2) -- (1.5,0.8);
\draw (1,-0.5) node{} -- (1,0) -- (1.5,0.2) node {} -- (2,0) -- (2,-0.5);
\draw (1,1.5) -- (1,1) -- (1.5,0.8) node {} -- (2,1) -- (2,1.5);
}
\buildrel\eqref{eq.Flaw}\over=
\Graph{\draw (2.5,1)--(2.5,1.5);\draw (1.5,0)--(1.5,-0.5) node{};\draw (3,0) -- (3,0.8) -- (2.5,1) node{} -- (2,0.8) -- (2,0.2) -- (1.5,0) node{} -- (1,0.2) -- (1,1.5);},
=
\Graph{\draw(2.5,1)--(2.5,1.5);\draw (3,0) -- (3,0.8) -- (2.5,1) node{} -- (2,0.8) -- (2,0.2) -- (1.5,0) node{} -- (1,0.2) -- (1,1.5);}
$$

Finally, item (2b) is obtained just by applying $\dag$ to (2a).
\end{proof}

We can characterize $*$-homomorphisms in a diagrammatic way (compare e.g.\ with \cite[Def.~2.3]{MRV19}):

\begin{prop}\label{P.homo}
Let $A,B$ be Frobenius algebras. Consider a linear map $f\colon A\to B$. Then
\begin{enumerate}
\item $f$ is multiplicative if and only if
$\displaystyle
\Graph{
\draw (1,0) -- (1.5,-0.2) -- (1.5,-0.5);
\draw (1,1.5) -- (1,0.8) node[rectangle]{$\scriptstyle f$} -- (1,0) node{} -- (0.5,-0.2) -- (0.5,-0.5);
}
=
\Graph{
\draw (1,-.5) -- (1,0.3) node[rectangle]{$\scriptstyle f$} -- (1,1) -- (1.5,1.2);
\draw (2,-.5) -- (2,0.3) node[rectangle]{$\scriptstyle f$} -- (2,1) -- (1.5,1.2) node {} -- (1.5,1.5);
}
$
\item $f$ is unital if and only if
$\Graph{\draw (1,1.3) -- (1,0.5) node[rectangle]{$\scriptstyle f$} -- (1,-0.3) node{};}
=\Graph{\draw (1,1) -- (1,0) node{};}$
\item $f$ is $*$-preserving if and only if
$\Graph{\draw (1,1.5) -- (1,0.5) node[rectangle]{$\scriptstyle f^\dag$} -- (1,-0.5);}
=\Graph{\draw (1,1.5) -- (1,-0.3) -- (1.5,-0.5) node{} -- (2,-0.3) -- (2,0.5) node[rectangle] {$\scriptstyle f$} -- (2,1.3) -- (2.5,1.5) node{} -- (3,1.3) -- (3,-0.5);}$
\end{enumerate}
\end{prop}
\begin{proof}
The first two items are completely straightforward. For instance, $f$ is multiplicative if $f(ab)=f(a)f(b)$, which can also be written as $f\circ m=m\circ(f\otimes f)$, which is exactly the meaning of the stated diagram. Similarly, $f$ is unital if and only if $f\circ\eta=\eta$.

The last item is slightly more involved. First, using Corollary~\ref{C.Flaw}(1), we see that the diagrammatic equation is equivalent to
$\Graph{\draw (1,-0.5) -- (1,1.3) -- (1.5,1.5) node{} -- (2,1.3) -- (2,0.5) node[rectangle]{$\scriptstyle f^\dag$} -- (2,-0.5);}
=\Graph{\draw (1,-0.5) -- (1,0.5) node[rectangle] {$\scriptstyle f$} -- (1,1.3) -- (1.5,1.5) node{} -- (2,1.3) -- (2,-0.5);}$. We can now study both sides of the equation applied on some elements $a\in A$, $b\in B$:
\begin{align*}
R^\dag(a\otimes f^\dag(b))&=\langle a^*,f^\dag(b)\rangle=\langle f(a^*),b\rangle\\
R^\dag(f(a)\otimes b)&=\langle f(a)^*,b\rangle
\end{align*}

Those are equal for every $a\in A$, $b\in B$ if and only if $f(a)^*=f(a^*)$, that is, $f$ is $*$-preserving.
\end{proof}

\subsection{Quantum graphs}
Now, we are ready for the definition of a quantum graph. Recall that a classical graph is determined by a finite set of vertices $X$ and a set of edges $E\subset X\times X$. There are several equivalent approaches for defining quantum graphs. The origins of quantum graphs go back to the work of Weaver \cite{Wea12,Wea21}. Later in \cite{MRV18}, an equivalent definition was formulated. We believe that the most intuitive definition of a quantum graph is the following one coming from \cite[Section~VII]{MRV18}, which directly generalizes the original definition of a graph.

\begin{defn}\label{D.qgraph}
A quantum graph is a pair $(X,\tilde A)$ of a finite quantum set $X$ and an orthogonal projection $\tilde A\in C(X)\otimes C(X)^{\rm op}$, $\tilde A^2=\tilde A$, $\tilde A^*=\tilde A$ called the \emph{edge projection}. The quantum graph
\begin{enumerate}
\item is \emph{undirected} if $\tilde A=T_{\crosspart}(\tilde A)$, where $T_{\crosspart}$ is the swap map $x\otimes y\mapsto y\otimes x$,
\item \emph{has no loops} if $\tilde A\tilde I=0$, where $\tilde I$ equals to $R$ taken as an element of $C(X)\otimes C(X)^{\rm op}$,
\item[(2*)] \emph{has a loop at every vertex} if $\tilde A\tilde I=\tilde I$.
\end{enumerate}
If conditions (1) and (2) are satisfied, we say that $X$ is a \emph{simple} quantum graph.
\end{defn}

\begin{rem}\label{R.op}
By $C(X)^{\rm op}$, we denote the \emph{opposite algebra}. That is, an algebra with the same underlying vector space as $C(X)$, but with opposite multiplication. It is well known that actually $C(X)^{\rm op}\simeq C(X)$ for any C*-algebra $C(X)$. For instance, if $C(X)$ is realized through matrices as $C(X)=\bigoplus_i M_{n_i}(\C)$, then the isomorphism $C(X)\to C(X)^{\rm op}$ can be realized as the transposition. So, if we have such a concrete realization of $C(X)$, we can use $\tilde A\in C(X)\otimes C(X)$.
\end{rem}

\begin{rem}
For a classical graph with a vertex set $X_N=\{1,\dots,N\}$, we take $\tilde A\in C(X_N)\otimes C(X_N)=C(X_N\times X_N)$ the indicator function of the set of edges, i.e.\ $\tilde A(v,w)=1$ if $(v,w)$ is an edge, $\tilde A(v,w)=0$ otherwise. 

A classical graph is undirected if $\tilde A(v,w)=\tilde A(w,v)$, which exactly corresponds to the condition (1). We have $R^{vw}=\delta_{vw}$, so the function $\tilde I$ is defined by $\tilde I(v,w)=\delta_{vw}$. Therefore, condition (2) means $\tilde A(v,w)\delta_{vw}=\tilde A(v,v)=0$ for every $v$, which exactly corresponds to the fact that the graph has no loops. In contrast, condition (2*) says $\tilde A(v,w)\delta_{vw}=\tilde A(v,v)=1$, which means that the graph has a loop at every vertex $v$.
\end{rem}

\begin{rem}
The notion \emph{graph} often means a simple graph. Similarly in the theory of quantum graphs, one often restricts only to the simple graphs. Nevertheless, in the current literature \cite{Wea12,Wea21,MRV18,MRV19,BCE+20,BEVW20}, the authors assume (2*) instead of (2). However, similarly to the classical case, this change has typically little practical impact.

To be more concrete, in Example~\ref{E.full}, we are going to show that $\tilde I$ is also an orthogonal projection and hence it follows that there is a one-to-one correspondence between graphs with no loops and graphs with a loop at every vertex: we simply have to add or subtract $\tilde I$ from $\tilde A$.
\end{rem}

\subsection{Adjacency matrix}

\begin{defn}
For any element $\tilde A\in C(X)\otimes C(X)^{\rm op}$, we define a linear operator $A\colon l^2(X)\to l^2(X)$ by $A=(\id\otimes R^\dag)(\tilde A\otimes\id)$. Conversely, for any $A\colon l^2(X)\to l^2(X)$, we define $\tilde A\in C(X)\otimes C(X)^{\rm op}$ by $\tilde A=(A\otimes\id)R$. Pictorially,
$$A=
\Graph{
\draw (1,1.5) -- (1,0.5) -- (1.5,0.3) node[rectangle]{$\scriptstyle\;\tilde A\;$} -- (2,0.5) -- (2,1) -- (2.5,1.2) node{} -- (3,1) -- (3,0);
}
,\qquad
\tilde A=
\Graph{
\draw (1,1.5) -- (1,0.5) node[rectangle]{$\scriptstyle A$} -- (1,0) -- (1.5,-0.2) node{} -- (2,0) -- (2,1.5);
}.
$$
From Corollary~\ref{C.Flaw}, it follows that those operations are inverse to each other. We will refer to them as the \emph{rotation}. We will say that $A$ is a \emph{rotated version} of $\tilde A$ and vice versa.
\end{defn}

\begin{defn}
Let $(X,\tilde A)$ be a quantum graph. The rotation $A=(\tilde A\otimes \id)R$ will be called the \emph{adjacency matrix} of the quantum graph $(X,\tilde A)$.
\end{defn}

\begin{ex}
Let $(X,\tilde A)$ be a classical graph. Then $R_{ij}=\delta_{ij}$, so
$$A^i_j=\sum_k \tilde A^{ik}R_{kj}=\tilde A^{ij}=\begin{cases}1&\text{if $i$ is adjacent to $j$,}\\0&\text{otherwise.}\end{cases}$$
\end{ex}

\begin{ex}\label{E.MnAdj}
In the case $X=M_n$, it is convenient to fix the isomorphism $t\colon C(X)\to C(X)^{\rm op}$ as the transposition and consider $\tilde A\in C(X)\otimes C(X)$ (by transposing the second leg). Recall that $R_{abcd}=n\delta_{ad}\delta_{bc}$. Using this convention, we can write $A=(\id\otimes(R^\dag t))\tilde A$, so
\begin{equation}\label{eq.rot}
A^{ij}_{kl}=\sum_{a,b}\tilde A^{ijba}R_{abkl}=n\tilde A^{ijkl}.
\end{equation}
\end{ex}

\begin{defn}
The rotation operation $C(X)\otimes C(X)^{\rm op}\to \Lin(l^2(X))$ induces the structure of a C*-algebra on $\Lin(l^2(X))$ -- the vector space of linear operators $l^2(X)\to l^2(X)$. For given $A,B\colon l^2(X)\to l^2(X)$, we will denote by $A\bullet B$ the corresponding product and call it the \emph{Schur product} of $A$ and $B$. The corresponding involution will be denoted by $A^*$.
\end{defn}

\begin{prop}\label{P.Schur}
It holds that
\begin{align*}
A\bullet B&=m(A\otimes B)m^\dag=
\Graph{
\draw (1,-1)--(1,-0.5)--(0,0)--(0,0.5) node[rectangle]{$\scriptstyle A$}--(0,1)--(1,1.5)--(1,2);
\draw (1,-0.5) node {}--(2,0)--(2,0.5) node[rectangle]{$\scriptstyle B$}--(2,1)--(1,1.5) node{};
}\\
A^*&=(R^\dag\otimes\id)(\id\otimes A^\dag\otimes\id)(\id\otimes R)=
\Graph{\draw (1,-0.5) -- (1,1.3) -- (1.5,1.5) node{} -- (2,1.3) -- (2,0.5) node[rectangle] {$\scriptstyle A^\dag$} -- (2,-0.3) -- (2.5,-0.5) node{} -- (3,-0.3) -- (3,1.5);}
\end{align*}
\end{prop}
\begin{proof}
For the purpose of this proof, denote by $\bullet$ and $*$ the operations given by the equations above. We want to prove that those are exactly the operations induced by the C*-structure on $C(X)\otimes C(X)^{\rm op}$.

That is, we want to prove that
$$((A\bullet B)\otimes\id)R=((A\otimes\id)R)\cdot((B\otimes\id)R),\qquad (A^*\otimes\id)R=((A\otimes\id)R)^*,$$
where the operations $\cdot$ and $*$ on the right hand side are the C*-operations in $C(X)\otimes C(X)^{\rm op}$. The first equation is easy to prove using diagrams. We want to prove that
$$
\Graph{
\draw (3,2)--(3,-0.8)--(2,-1) node{}--(1,-0.8)--(1,-0.5)--(0,0)--(0,0.5) node[rectangle]{$\scriptstyle A$}--(0,1)--(1,1.5)--(1,2);
\draw (1,-0.5) node {}--(2,0)--(2,0.5) node[rectangle]{$\scriptstyle B$}--(2,1)--(1,1.5) node{};
}
=
\Graph{
\draw (2,1.5) -- (1,1) -- (1,0.5) node[rectangle] {$\scriptstyle A$} -- (1,-0.3) -- (1.5,-0.5) node{} -- (2,-0.3) -- (2,1) -- (4,1.3) -- (3.5,1.5);
\draw (2,2) -- (2,1.5) node{} --(3,1) -- (3,0.5) node[rectangle] {$\scriptstyle B$} -- (3,-0.3) -- (3.5,-0.5) node{} -- (4,-0.3) -- (4,1) -- (3,1.3) -- (3.5,1.5) node {} -- (3.5,2);
}.
$$
Note the crossing under one of the multiplications $\Gfork$, which depicts the fact that we use the opposite multiplication in the right factor.

It is easy to see that the right hand side actually equals to 
$
\Graph{
\draw (2,1.5) -- (1,1) -- (1,0.5) node[rectangle] {$\scriptstyle A$} -- (1,-0.3) -- (3,-0.8) node{} -- (5,-0.3) -- (5,1.3) -- (4.5,1.5);
\draw (2,2) -- (2,1.5) node{} --(3,1) -- (3,0.5) node[rectangle] {$\scriptstyle B$} -- (3,0) -- (3.5,-0.2) node{} -- (4,0) -- (4,1.3) -- (4.5,1.5) node {} -- (4.5,2);
}$.
Now, we can use Corollary~\ref{C.Flaw} to see that 
this indeed equals to
\Graph{
\draw (3,2)--(3,-0.8)--(2,-1) node{}--(1,-0.8)--(1,-0.5)--(0,0)--(0,0.5) node[rectangle]{$\scriptstyle A$}--(0,1)--(1,1.5)--(1,2);
\draw (1,-0.5) node {}--(2,0)--(2,0.5) node[rectangle]{$\scriptstyle B$}--(2,1)--(1,1.5) node{};
}.

The second equation is not as easy to prove using diagrams so let us do it in entries. Take some orthonormal basis $(e_i)$ of $l^2(X)$. Recall from Lemma~\ref{L.RF} that if we denote $e_i^*=\sum_jF_i^je_j$, then $F_i^j=R^{ij}$. Recall also Corollary~\ref{C.Flaw}, where the first equation can be written as $\sum_k R^\dag_{ik}R^{kj}=\delta_{ij}=\sum_k R^{ik}R^\dag_{kj}$. Finally, recall that we are working over symmetric Frobenius algebras, so $R^{ij}=R^{ji}$. Now, we can write
\begin{align*}
[(A^*\otimes\id)R]^{ij}&=\sum_k [A^*]^i_kR^{kj}=\sum_{k,a,b}R^{ai}[A^\dag]^b_aR^\dag_{kb}R^{kj}=\sum_a  R^{ai}\overline{A^a_j},\\
[((A\otimes\id)R)^*]^{ij}&=\sum_{a,b}F^i_aF^j_b\overline{[(A\otimes\id)R]^{ab}}=\sum_{k,a,b}R^{ia}R^{jb}\overline{R^{kb}}\overline{A^a_k}=\sum_aR^{ia}\overline{A^a_j}.
\end{align*}
Both sides are equal, which is what we wanted to show.
\end{proof}

\begin{ex}\label{E.full}
Let $X$ be any finite quantum set. The element $\tilde I=R\in C(X)\otimes C(X)^{\rm op}$ is exactly the rotation of the identity $I:=\id\colon l^2(X)\to l^2(X)$. By definition of a finite quantum set, the identity satisfies $I\bullet I=m(I\otimes I)m^\dag=mm^\dag=I$. From Corollary \ref{C.Flaw}, we also have $I^*=I$. Consequently, $\tilde I$ is an orthogonal projection and hence defines a quantum graph.

We will also denote $\tilde J:=1_{C(X)}\otimes 1_{C(X)}$, which is obviously a projection, so it also defines a quantum graph. Note that $R^\dag(\eta\otimes\id)=\eta^\dag$ by definition of $R^\dag$, so $J=(\id\otimes R^\dag)(\eta\otimes\eta\otimes\id)=\eta\eta^\dag$.

More concretely, we define the following quantum graphs over $X$:
\begin{itemize}
\item The \emph{empty graph} with no loops $\tilde A=0$,
\item The \emph{empty graph} with a loop at every vertex $\tilde A=\tilde I$,
\item The \emph{full graph} with no loops $\tilde A=\tilde J-\tilde I$,
\item The \emph{full graph} with a loop at every vertex $\tilde A=\tilde J$.
\end{itemize}

This is a very natural definition: The zero projection $\tilde A=0$ just means that there are no edges. In contrast, if we take the algebra identity $\tilde A=\tilde J=1_{C(X)}\otimes 1_{C(X)}$, then this means we have a projection on the whole set $X\times X$, so every pair of vertices is an edge. The identity matrix $I$ as adjacency matrix corresponds to loops in the classical case and we interpret it this way also in the quantum case.

The element $\tilde J$ is obviously symmetric and since we work over symmetric Frobenius algebras, the element $\tilde I$ must be symmetric as well. Consequently, all the graphs above are undirected.
\end{ex}

\begin{prop}[Equivalent definition of a quantum graph]
Let $X$ be a finite quantum set. A linear operator $A\colon l^2(X)\to l^2(X)$ is an adjacency matrix of a quantum graph if and only if $A\bullet A=A$, $A=A^*$. The quantum graph
\begin{enumerate}
\item is undirected if and only if $A=A^\dag$,
\item has no loops if and only if $A\bullet I=0$,
\item[(2*)] has a loop at every vertex if $A\bullet I=I$.
\end{enumerate}
\end{prop}
\begin{proof}
The only point that needs a proof is the item (1). Everything else is a direct consequence of Proposition~\ref{P.Schur}.

In order to prove (1), it is enough to show that the rotation of $A^\dag$ equals to $T_{\crosspart}(\tilde A)$. This is indeed true: Recall from Proposition~\ref{P.Schur} that $A=A^*=
\Graph{\draw (1,-0.5) -- (1,1.3) -- (1.5,1.5) node{} -- (2,1.3) -- (2,0.5) node[rectangle] {$\scriptstyle A^\dag$} -- (2,-0.3) -- (2.5,-0.5) node{} -- (3,-0.3) -- (3,1.5);}
$.
Consequently,
\[\widetilde{A^\dag}=
\Graph{\draw (2,1.5) -- (2,0.5) node[rectangle] {$\scriptstyle A^\dag$} -- (2,-0.3) -- (2.5,-0.5) node{} -- (3,-0.3) -- (3,1.5);}
= 
\Graph{\draw (0,1.5) -- (0,-0.3) -- (0.5,-0.5) node{} -- (1,-0.3) -- (1,1.3) -- (1.5,1.5) node{} -- (2,1.3) -- (2,0.5) node[rectangle] {$\scriptstyle A^\dag$} -- (2,-0.3) -- (2.5,-0.5) node{} -- (3,-0.3) -- (3,1.5);}
=
\Graph{\draw (0,1.5) -- (0,-0.3) -- (0.5,-0.5) node{} -- (1,-0.3) -- (1,0.5) node[rectangle]{$\scriptstyle A$} -- (1,1.5);}
=
\Graph{\draw (3,1.5) -- (3,1.4) -- (2,1.2) -- (2,0.4) node[rectangle] {$\scriptstyle A$} -- (2,-0.3) -- (2.5,-0.5) node{} -- (3,-0.3) -- (3,1.2) -- (2,1.4) -- (2,1.5);}
=T_{\crosspart}(\tilde A).
\qedhere\]
\end{proof}

This is how quantum graphs are defined in \cite{MRV18,MRV19,BCE+20}. The equivalence of the two definitions was originally proven in \cite[Theorem~VII.7]{MRV18}. In \cite{BEVW20}, the authors also omit the condition $A=A^*$.

\begin{ex}[Anticommutative square]\label{E.square}
We would like to construct an example of some really non-trivial quantum graph. So, consider the simplest non-trivial finite quantum set, which is $M_2$. Recall from Remark~\ref{R.op} and Example~\ref{E.MnAdj} that considering an edge projection $\tilde A$ on $M_2$, it is convenient to transpose the second leg so that we have $\tilde A\in C(M_2)\otimes C(M_2)=M_2(\C)\otimes M_2(\C)\simeq M_4(\C)$. We will use this convention throughout this example.

So, in order to define a quantum graph on $M_2$, we just need to pick some projection in $M_4(\C)$.\footnote{This actually makes it quite easy to classify all such graphs, which we are going to do in Section~\ref{secc.smallclass}.} Let us pick for instance
$$\tilde A=\frac{1}{2}
\begin{pmatrix}
1&0&0&-1\\
0&1&1&0\\
0&1&1&0\\
-1&0&0&1
\end{pmatrix}.
$$

This is obviously an orthogonal projection, so it indeed defines the structure of a quantum graph on $M_2$. Using formula~\eqref{eq.rot}, we can compute the corresponding adjacency matrix $A\colon l^2(M_2)\to l^2(M_2)$. Actually, the adjacency matrix can be read of the form of $\tilde A$ directly, if we write it down more suggestively as an element of $M_2(\C)\otimes M_2(\C)\simeq M_2(M_2(\C))$:
$$\tilde A=
\frac{1}{2}\begin{pmatrix}
\begin{pmatrix}1&0\cr0&1\end{pmatrix}&
\begin{pmatrix}0&-1\cr1&0\end{pmatrix}\cr
\begin{pmatrix}0&1\cr-1&0\end{pmatrix}&
\begin{pmatrix}1&0\cr0&1\end{pmatrix}
\end{pmatrix}.
$$

One can check that the small matrices actually stand for the images of the standard basis $(e_{ij})$ under the adjacency operator $A$. That is, the adjacency matrix $A\colon l^2(M_2)\to l^2(M_2)$ is given by
\begin{align*}
\begin{pmatrix}1&0\\0&0\end{pmatrix}\mapsto\begin{pmatrix}1&0\\0&1\end{pmatrix},\qquad
&\begin{pmatrix}0&1\\0&0\end{pmatrix}\mapsto\begin{pmatrix}0&-1\\1&\phantom+0\end{pmatrix},\cr
\begin{pmatrix}0&0\\1&0\end{pmatrix}\mapsto\begin{pmatrix}\phantom+0&1\\-1&0\end{pmatrix},\qquad
&\begin{pmatrix}0&0\\0&1\end{pmatrix}\mapsto\begin{pmatrix}1&0\\0&1\end{pmatrix}.
\end{align*}

Alternatively, we can express the images of the standard basis $(e_{ij})$ as columns and write $A$ really as a matrix
$$A=\begin{pmatrix}1&0&0&1\\0&-1&1&0\\0&1&-1&0\\1&0&0&1\end{pmatrix}.$$

Now, it is obvious that $A=A^\dag$, so the corresponding graph is undirected. (We are using the fact that $(e_{ij})$ is an orthogonal basis, so the adjoint $\dag$ is just conjugate transpose.)

Reversing this procedure, we can easily determine the edge projection $\tilde I$ corresponding to the empty graph with a loop at every vertex
$$I=
\begin{pmatrix}
1&0&0&0\\
0&1&0&0\\
0&0&1&0\\
0&0&0&1
\end{pmatrix}
\quad\to\quad
\tilde I=
\frac{1}{2}
\begin{pmatrix}
\begin{pmatrix}1&0\cr0&0\end{pmatrix}&
\begin{pmatrix}0&1\cr0&0\end{pmatrix}\cr
\begin{pmatrix}0&0\cr1&0\end{pmatrix}&
\begin{pmatrix}0&0\cr0&1\end{pmatrix}
\end{pmatrix}
=
\frac{1}{2}
\begin{pmatrix}
1&0&0&1\\
0&0&0&0\\
0&0&0&0\\
1&0&0&1
\end{pmatrix}.
$$

Finally, a straightforward computation shows that $\tilde A\tilde I=0$, so our graph $(M_2,\tilde A)$ has no loops.
\end{ex}

\section{Basic quantum graph constructions}
\label{sec.constructions}

\subsection{Isomorphism of quantum graphs}

In this section, we define the notion of an isomorphism for quantum graphs. We work here within the classical group setting. Isomorphisms of quantum graphs should not be confused with quantum isomorphisms, which will be defined in Def.~\ref{D.qiso}.

\begin{defn}
Let $X$, $Y$ be finite quantum spaces. We say that $X$ and $Y$ are \emph{isomorphic}, denoted by $X\simeq Y$, if there is a $*$-isomorphism $\phi\colon C(X)\to C(Y)$.
\end{defn}

This $*$-isomorphism then in particular must preserve the unique tracial $\delta$-form and hence it preserves the counit $\eta^\dag$, the inner product $\langle\cdot,\cdot\rangle$, and the bilinear form $R^\dag$.

We define \emph{automorphisms} of $X$ in the natural way: as isomorphisms $X\to X$. Naturally, they form a group, which will be denoted by $\Aut X$. Again, this should not be confused with the quantum automorphism group $\Aut^+ X$, which will be defined in Definitions \ref{D.QAut0}, \ref{D.QAut}.

\begin{prop}\label{P.Aut}
$\Aut M_n=PU_n$.
\end{prop}
\begin{proof}
It is well known that any automorphism $M_n(\C)\to M_n(\C)$ must be inner, i.e.\ of the form $x\mapsto UxU^{-1}$ for some invertible matrix $U\in GL_n$ (it can be proven for instance as a consequence of Skolem--Noether theorem). In addition, we require it to be a $*$-automorphism, which implies that $U$ must be unitary. That is, we have an action of $U_n$ on $M_n(\C)$ by conjugation, which realizes all automorphisms of $M_n(\C)$. As follows from Schur's lemma, the kernel of this action are precisely scalar multiples of the identity matrix, so the automorphism group equals to $U_n/\{\lambda I\}=PU_n$.
\end{proof}

\begin{rem}
In \cite{Wan98}, Wang mentions that the automorphism group of $M_n$ is the special unitary group $SU_n$. This is technically not entirely correct since $SU_n$ does not act faithfully on $M_n(\C)$. Precisely as we mentioned in the proof, we have $\lambda I\in SU_n$, where $\lambda$ is some $n$-th root of unity, which acts on $M_n(\C)$ trivially.
\end{rem}

\begin{defn}
Let $(X,\tilde A_X)$, $(Y,\tilde A_Y)$ be quantum graphs. We say that they are isomorphic, denoted by $(X,\tilde A_X)\simeq(Y,\tilde A_Y)$, if there is a $*$-isomorphism $\phi\colon C(X)\to C(Y)$ such that $\phi\circ A_X=A_Y\circ\phi$. An automorphism of $(X,\tilde A)$ is an isomorphism $(X,\tilde A)\to(X,\tilde A)$. We denote by $\Aut(X,\tilde A)$ the group of automorphisms of $(X,\tilde A)$.
\end{defn}

\begin{rem}
For the full graph and the empty graph, with or without loops, the automorphism group coincides with automorphisms of the underlying quantum space. Indeed, any $*$-automorphism $C(X)\to C(X)$ automatically commutes with $I$ and $J=\eta\eta^\dag$.
\end{rem}

\subsection{Quantum subgraphs}

A subgraph of a classical graph can be made either by removing edges or by removing vertices. For simplicity, let us first discuss those two constructions separately.

\begin{defn}\label{D.Esubgraph}
Let $(X,\tilde A)$ be a quantum graph. We say that a quantum graph $(X,\tilde A')$ is a quantum subgraph \emph{made by removing edges} if $\tilde A'\le\tilde A$ as a projection.
\end{defn}


Now, we look on the procedure of removing vertices.

\begin{defn}\label{D.qsubspace}
Let $X$ be a finite quantum set. A \emph{quantum subset} $Y\subset X$ is the finite quantum set given by any quotient $C(Y)$ of $C(X)$.
\end{defn}

\begin{rem}\label{R.qsubspace}
Recall that any finite-dimensional C*-algebra is of the form $C(X)=\bigoplus_{i=1}^a M_{n_i}(\C)$. It is well known that matrix algebras $M_n(\C)$ are simple. Hence, taking a quotient $C(Y)$ of $C(X)$ amounts to sending some of the factors $M_{n_i}(\C)$ to zero. In other words, taking a quantum subset $Y\subset X$ amounts to taking a classical subset $S\subset \{1,\dots,a\}$.

As a consequence, considering the surjective $*$-homomorphism $q\colon C(X)\to C(Y)$, it holds that $q^\dag$ is an isometry and a non-unital $*$-homomorphism embedding the non-zero factors in $C(Y)$ to $C(X)$.
\end{rem}

\begin{defn}\label{D.Vsubgraph}
Let $(X,\tilde A_X)$ be a quantum graph. Consider a finite quantum subset $Y\subset X$. We define the \emph{induced subgraph} $(Y,\tilde A_Y)$ to be given by $\tilde A_Y=(q\otimes q)(\tilde A_X)$, where $q\in C(X)\to C(Y)$ is the corresponding surjective $*$-homomorphism.
\end{defn}

Since $q$ is a $*$-homomorphism, it is easy to check that $\tilde A_Y$ indeed is an orthogonal projection. Using Proposition~\ref{P.homo}(3), we can write the definition in a rotated way using the adjacency matrices as $A_Y=qA_Xq^\dag$.

Summarizing the definitions above, we may define a general notion of a quantum subgraph:

\begin{defn}\label{D.subgraph}
Let $(X,\tilde A_X)$, $(Y,\tilde A_Y)$ be quantum graphs. We say that $(X,\tilde A_X)$ is a \emph{quantum subgraph} of $(Y,\tilde A_Y)$ if there is a surjective $*$-homomorphism $C(X)\to C(Y)$ such that $\tilde A_Y\le (q\otimes q)(\tilde A_X)$.
\end{defn}

\begin{rem}
In \cite[Section~6]{Wea21} a similar but weaker notion named \emph{restriction} is defined.
\end{rem}

\subsection{Remark on vertices, edges and counting them}\label{secc.vertices}

In the sense of Definition~\ref{D.qsubspace} and Remark~\ref{R.qsubspace}, it seems natural to interpret the factors in the decomposition $C(X)=\bigoplus_{i=1}^a M_{n_i}(\C)$ as \emph{points} in the finite quantum space or \emph{vertices} of the quantum graph. Indeed, those are constituents of the quantum space $X$ that are indecomposable using the notion of a quantum subset. One might be then tempted to say that the finite quantum space $X$ consists of $a$ vertices.

However, as will turn out later on, this is not a convenient viewpoint. We claim that the factor $M_n(\C)$ should be considered to be consisting of $n^2$ vertices, so the whole finite quantum space $X$ with $C(X)=\bigoplus_i M_{n_i}(\C)$ consists of $\#X:=\sum_i n_i^2=\dim C(X)$ vertices. One reason is that this quantity can be expressed in a sort of invariant way as $\#X=\eta^\dag\eta$. In particular, it will turn out that the number of vertices is invariant under \emph{quantum isomorphisms}, which will be defined in Def.~\ref{D.qiso}. Another reason comes in the following section, where we are going to define \emph{quantum quotient sets}. A reasonable requirement is that taking quantum quotients always reduces the number of vertices, which would not be true for the former definition of a vertex.%
\footnote{Fans of chemistry might enjoy the following parallel: The finite quantum sets $X$ are molecules, which consist of atoms $M_{n_i}(\C)$. The atoms $M_{n_i}$ are defined by their proton number $n_i$, they have atomic mass $n_i^2$, and they cannot be decomposed by any chemical processes (taking quantum subsets), but can be split by nuclear fission (taking quantum set quotient).}

So, from now on, given a quantum graph $(X,\tilde A)$ with $C(X)=\bigoplus_{i=1}^a M_{n_i}(\C)$ we will distinguish between the following two notions. A \emph{quantum vertex} is the indecomposable component $M_{n_i}$ of $X$, where we need to keep in mind that the number of quantum vertices $a$ is not a nice invariant notion. In contrast, the notion \emph{vertex} will not have any concrete meaning for quantum graphs, but we define the \emph{number of vertices} $\#V=\eta^\dag\eta=\dim C(X)$. In particular, we say that each quantum vertex $M_{n_i}$ consists of $n_i^2$ vertices.

A similar story can be said about edges. Let $(X,\tilde A)$ be a quantum graph. Indecomposable constituents of the edge projection $\tilde A$ are the minimal projections $\tilde P\le \tilde A$. So, in a sense, those could be considered as the edges of $(X,\tilde A)$ (we will sometimes refer to them as the \emph{quantum edges}). Nevertheless, when counting edges, one should consider a slightly more sophisticated viewpoint defining the \emph{number of edges} as $\#E(V,\tilde A):=(\eta^\dag\otimes\eta^\dag)(\tilde A)=\eta^\dag A\eta$. In particular, taking a minimal projection $\tilde P\in M_{n_i}(\C)\otimes M_{n_j}(\C)$, we interpret it as consisting of $n_in_j$ edges.

Note that in \cite{Mat21} a definition of {\em $d$-regular} quantum graphs was introduced using the condition $\eta^\dag A=d\eta^\dag$. This means that any $d$-regular graph has $\eta^\dag A\eta=d\eta^\dag\eta=dN$ edges, where $N=\eta^\dag\eta$ is the number of vertices. This exactly corresponds to the classical case.

\subsection{Graph quotients}

\begin{defn}
Let $X$ be a finite quantum set. A \emph{quotient quantum set} $Y$ is the finite quantum set given by a $*$-subalgebra $C(Y)\subset C(X)$.
\end{defn}

We aim to define a quotient quantum graph by putting $\tilde A_Y=(\iota^\dag\otimes\iota^\dag)(\tilde A_X)$. Note however that even classical graphs are not closed under this procedure, because taking such a quotient may create multiple edges. Therefore, we need to extend our framework slightly.

\begin{defn}
A \emph{weighted quantum graph} is a pair $(X,\tilde A)$, where $X$ is a finite quantum set and $\tilde A\in C(X)\otimes C(X)^{\rm op}$ a positive element. A quantum \emph{multigraph} is a weighted quantum graph, where the spectrum of $\tilde A$ is contained in $\N_0$. For both cases we define the notion of being \emph{undirected} and having \emph{no loops} the same way as in the case of quantum graphs.
\end{defn}

Recall the following well-known fact.

\begin{lem}\label{L.positive}
A linear map $\phi\colon C(X)\to C(Y)$ is positive if and only if $\phi^\dag\colon C(Y)\to C(X)$ is positive.
\end{lem}
\begin{proof}
As usual, $\dag$ denotes the adjoint with respect to the inner product in $l^2(X)$ and $l^2(Y)$. Suppose that $C(X)=\bigoplus_i M_{n_i}(\C)$, $C(Y)=\bigoplus_j M_{m_j}(\C)$. Then $\phi$ decomposes as a sum of mappings $\phi^j_i\colon M_{n_i}(\C)\to M_{m_j}(\C)$. One can see that all $\phi_i^j$ are positive if and only if $\phi$ itself is positive.

Consequently, we may without loss of generality assume that $C(X)=M_n(\C)$ and $C(Y)=M_m(\C)$. The mapping $\phi$ is positive if it maps positive elements of $C(X)$ to positive elements of $C(Y)$. A positive element can be diagonalized and hence written as a linear combination of projections. Therefore, $\phi$ is positive if and only if $\phi(xx^*)\ge 0$ for every $x\in\C^n$. An element $b\in C(Y)$ is positive if and only if $y^*by\ge 0$ for every $y\in\C^m$. So, $\phi$ is positive if and only if
\begin{align*}
0\le y^*\phi(xx^*)y&=\Tr(yy^*\phi(xx^*))=\frac{1}{m}\langle yy^*,\phi(xx^*)\rangle=\frac{1}{m}\langle \phi^\dag(yy^*),xx^*\rangle\\&=\frac{n}{m}\,\overline{\Tr(xx^*\phi^\dag(yy^*))}=\frac{n}{m}\,\overline{x^*\phi^\dag(yy^*)x},
\end{align*}
which holds if and only if $\phi^\dag$ is positive.
\end{proof}

Recall also that any $*$-homomorphism must be positive. So, in particular, any embedding $\iota\colon C(Y)\to C(X)$ is positive and, thanks to the lemma above, also its adjoint $\iota^\dag$ is positive. Consequently, considering a positive element $\tilde A_X\in C(X)\otimes C(X)^{\rm op}$, it must hold that $(\iota^\dag\otimes\iota^\dag)\tilde A_X$ is also positive. This allows to make the following definition.

\begin{defn}
Let $(X,\tilde A_X)$ be a weighted quantum graph. Consider $Y$ a quantum quotient space of $X$. We define the quantum quotient graph $(Y,\tilde A_Y)$ by $\tilde A_Y=(\iota^\dag\otimes\iota^\dag)(\tilde A_X)$, where $\iota$ is the corresponding embedding $C(Y)\to C(X)$.
\end{defn}

\begin{rem}
Since $\iota$ is a unital $*$-homomorphism, we have $\iota\eta_Y=\eta_X$. Consequently also $\eta_Y^\dag\iota^\dag=\eta_X^\dag$. Therefore, taking graph quotients preserves the quantity $(\eta^\dag\otimes\eta^\dag)\tilde A=\eta^\dag A\eta$, which in case of weighted quantum graphs can be interpreted as the \emph{sum of all weights}. As we mentioned in Section~\ref{secc.vertices}, for ordinary quantum graphs (or quantum multigraphs), this can be interpreted simply as the number of all edges.
\end{rem}

In contrast with classical graphs, in the quantum world taking graph quotients may not only create multiple edges, but we may obtain a general weighted graph (i.e.\ the multiplicities may not be natural numbers). We illustrate this in the following example.

\begin{ex}
Consider the quantum set $M_2$, which has the classical set of two elements $X_2$ as a quotient -- we may take the obvious inclusion $C(X_2)=\C\oplus\C\to M_2(\C)$ using diagonal matrices. Using the standard bases $(e_1,e_2)$ for $\C^2=C(X_2)$ and $(e_{11},e_{12},e_{21},e_{22})$ for $M_2(\C)$, we may express $\iota\colon l^2(X_2)\to l^2(M_2)$ as a matrix
$$\iota=
\begin{pmatrix}
1&0\\
0&0\\
0&0\\
0&1
\end{pmatrix}
$$

Note that the basis $(e_i)$ of $\C^2$ is orthonormal, but the basis $(e_{ij})$ of $M_2(\C)$ is not. Consequently, the $\iota^\dag$ is not given just by conjugate transpose, but we have to normalize it:
$$\iota^\dag=2
\begin{pmatrix}
1&0&0&0\\
0&0&0&1
\end{pmatrix}
$$

Now we define a graph $X=(M_2,\tilde A_X)$ on $M_2$ by
$$\tilde A_X=\frac{1}{4}
\begin{pmatrix}
1&1&1&-1\\
1&1&1&-1\\
1&1&1&-1\\
-1&-1&-1&1
\end{pmatrix}
\quad\leftrightarrow\quad
A_X=\frac{1}{2}
\begin{pmatrix}
1&1&1&1\\
1&-1&1&-1\\
1&1&-1&-1\\
1&-1&-1&1
\end{pmatrix}
$$

Note that this is a simple quantum graph, so the nicest possible example. In addition, $\tilde A_X$ is a minimal projection and it consists of $(\eta^\dag\otimes\eta^\dag)\tilde A_X=n^2\Tr\tilde A_X=4$ edges. Now, let us compute the quotient graph $Y=(X_2,\tilde A_Y)$:
$$A_Y=\iota^\dag A_X\iota=
\begin{pmatrix}
1&1\\
1&1
\end{pmatrix}
$$

This is an adjacency matrix of a classical graph with four simple edges, which can be drawn as
\begin{tikzpicture}[x=1.5em,y=1.5em, baseline=-3pt]
\node[circle,draw,fill=black,inner sep=1pt] (a) at (0,0) {};
\node[circle,draw,fill=black,inner sep=1pt] (b) at (1,0) {};
\path[->] (a) edge[bend left] (b) edge[loop left] ();
\path[->] (b) edge[bend left] (a) edge[loop right] ();
\end{tikzpicture}.
It is actually the full graph on two vertices.

However, modifying this example, we can arrive at some weighted classical graph instead. Take for instance
$$\tilde A_X=\frac{1}{8}
\begin{pmatrix}
     3& \sqrt3& \sqrt3&-     3\\
\sqrt3&      1&      1&-\sqrt3\\
\sqrt3&      1&      1&-\sqrt3\\
    -3&-\sqrt3&-\sqrt3&      3
\end{pmatrix}
\quad\leftrightarrow\quad
A_X=\frac{1}{4}
\begin{pmatrix}
     3& \sqrt3& \sqrt3&      1\\
\sqrt3&     -3&      1&-\sqrt3\\
\sqrt3&      1&     -3&-\sqrt3\\
     1&-\sqrt3&-\sqrt3&      3
\end{pmatrix}.
$$

This is again a minimal projection and defines a simple quantum graph with four edges. Its quotient is then given by the adjacency matrix
$$A_Y=\iota^\dag A_X\iota=
\frac{1}{2}
\begin{pmatrix}
3&1\\
1&3
\end{pmatrix}.
$$
That is, we get a weighted graph
\begin{tikzpicture}[x=1.5em,y=1.5em, baseline=-3pt]
\node[circle,draw,fill=black,inner sep=1pt] (a) at (0,0) {};
\node[circle,draw,fill=black,inner sep=1pt] (b) at (1,0) {};
\path[->] (a) edge[bend left] node[above] {$\scriptstyle1/2$} (b)
              edge[loop left] node {$\scriptstyle3/2$} ();
\path[->] (b) edge[bend left] node[below] {$\scriptstyle1/2$} (a)
              edge[loop right] node{$\scriptstyle3/2$} ();
\end{tikzpicture}.
\end{ex}

\section{Concrete viewpoint on quantum graphs}\label{sec.concrete}

In this section, we would like to build some intuition about what are actually quantum graphs and how to construct them. In the simplest case $X=M_2$, we are going to classify all quantum graphs up to isomorphism. The main tool for that is an equivalent definition of a quantum graph, which actually is the original definition formulated by Weaver in \cite{Wea12}. 

\subsection{General case}

\begin{figure}
\centering
\begin{tikzpicture}[x=1em, y=1em]
\node[circle,fill=black!25,label=left:$\scriptstyle V_{33}\subset\Lin{(\C^3,\C^3)}$](A) at (-7,0)
{
	\begin{tikzpicture}[x=1em, y=1em,every node/.style={circle,draw,fill=white,inner sep=1pt},]
	\draw[gray] (1,1) -- (1,2) -- (2,2) -- (2,1) -- cycle;
	\draw[gray] (2,2) -- (2,3) -- (3,3) -- (3,2) -- cycle;
	\draw[gray] (3,1) -- (1,3);
	\draw[gray] (3,2) -- (2,1);
	\draw[gray] (2,3) -- (1,2);
	\node[circle,fill=black] at (1,1){};
	\node[circle,fill=black] at (1,2){};
	\node[circle,fill=black] at (1,3){};
	\node[circle,fill=black] at (2,1){};
	\node[circle,fill=black] at (2,2){};
	\node[circle,fill=black] at (2,3){};
	\node[circle,fill=black] at (3,1){};
	\node[circle,fill=black] at (3,2){};
	\node[circle,fill=black] at (3,3){};
	\end{tikzpicture}
};
\node[circle,fill=black!25,label=right:$\scriptstyle V_{22}\subset\Lin{(\C^2,\C^2)}$](B) at (7,0)
{
	\begin{tikzpicture}[x=1em, y=1em,every node/.style={circle,draw,fill=white,inner sep=1pt},]
	\node[circle,fill=black] (a) at (1,1){};
	\node[circle,fill=black] (b) at (1,2){};
	\node[circle,fill=black] (c) at (2,1){};
	\node[circle,fill=black] (d) at (2,2){};
	\draw[gray,->] (a) -- (c);
	\draw[gray,->] (c) -- (d);
	\draw[gray,->] (d) -- (b);
	\draw[gray,->] (b) -- (a);
	\end{tikzpicture}
};
\node[circle,fill](C) at (0,-10){};
\path[->] (C) edge [loop below] node[below]{$\scriptstyle V_{11}\subset\Lin{(\C,\C)}$} ();
\fill[fill=black!25] (C.135) to[bend left] (A.270) to (A.225) to[bend right] node[sloped,below]{$\scriptstyle V_{13}\subset\Lin{(\C,\C^3)}$} (C.180) to cycle;
\draw[gray,->] (C.157) to[bend left] (A.230);
\draw[gray,->] (C.157) to[bend left] (A.248);
\draw[gray,->] (C.157) to[bend left] (A.265);
\fill[fill=black!25] (A.270) to[bend left] node[sloped,below]{$\scriptstyle V_{31}\subset\Lin{(\C^3,\C)}$} (C.135) to (C.90) to[bend right] (A.315) to cycle;
\draw[gray,->] (A.275) to[bend left] (C.112);
\draw[gray,->] (A.293) to[bend left] (C.112);
\draw[gray,->] (A.310) to[bend left] (C.112);
\fill[fill=black!25] (C.45) to[bend right] (B.270) to (B.315) to[bend left] node[sloped,below]{$\scriptstyle V_{12}\subset\Lin{(\C,\C^2)}$} (C.0) to cycle;
\draw[gray,->] (C.22) to[bend right] (B.280);
\draw[gray,->] (C.22) to[bend right] (B.305);
\fill[fill=black!25] (B.225) to[bend right] (C.90) to (C.45) to[bend left] node[sloped,below]{$\scriptstyle V_{21}\subset\Lin{(\C^2,\C)}$} (B.270) to cycle;
\draw[gray,->] (B.235) to[bend right] (C.67);
\draw[gray,->] (B.260) to[bend right] (C.67);
\fill[fill=black!25] (A.45) to[bend left] node[sloped,above]{$\scriptstyle V_{32}\subset\Lin{(\C^3,\C^2)}$} (B.135) to (B.180) to[bend right] (A.0) to cycle;
\draw[gray,->] (A.40) to[bend left] (B.145);
\draw[gray,->] (A.40) to[bend left] (B.170);
\draw[gray,->] (A.23) to[bend left] (B.145);
\draw[gray,->] (A.5) to[bend left] (B.170);
\fill[fill=black!25] (B.225) to[bend left] (A.315) to (A.0) to[bend right] node[sloped,above]{$\scriptstyle V_{23}\subset\Lin{(\C^2,\C^3)}$} (B.180) to cycle;
\draw[gray,->] (B.190) to[bend left] (A.320);
\draw[gray,->] (B.215) to[bend left] (A.355);
\draw[gray,->] (B.215) to[bend left] (A.337);
\draw[gray,->] (B.215) to[bend left] (A.320);
\end{tikzpicture}
\caption{A diagram of a quantum graph}\label{fig.qgraph}
\smallskip\footnotesize
This diagram is supposed to depict a typical quantum graph. In this case, we chose a quantum space $X$ with $C(X)=\C\oplus M_2(\C)\oplus M_3(\C)$. That is, the graph consists of one classical ``full-fat'' vertex and two quantum vertices $M_2$ and $M_3$, which consist of four resp.\ nine vertices. Each of these quantum vertices can have some ``internal edges'', which are described by the operator spaces $V_{ii}\subset \Lin(\C^i,\C^i)$. In particular, for the classical vertex, we have $V_{11}\subset\Lin(\C,\C)\simeq\C$ -- the potential internal edge is the loop, which either is or is not contained in the graph. In our case, we include the loop and denote it by the black loop arrow. The internal edges of the non-classical quantum vertices have so-to-say a quantum character and in the picture are illustrated by the gray edges inside the gray circles. In this case, the gray lines are purely illustrative and have little formal meaning. The same can be said about the spaces $V_{ij}\subset\Lin(\C^i,\C^j)$ that describe the oriented edges between the quantum vertices.

Finally, note that if $X$ was a classical space, so $C(X)=\C\oplus\C\oplus\C\oplus\cdots$, then such a diagram would constitute the standard picture of a classical graph. In contrast, if we have $X=M_n$, then trying to draw such a diagram is completely pointless.
\end{figure}
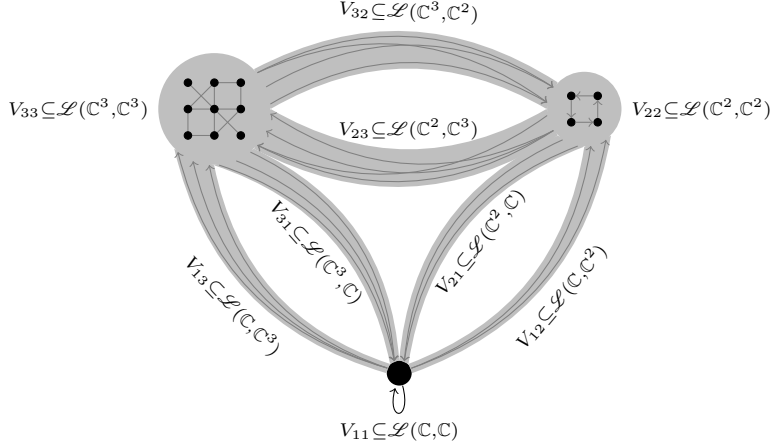

As the title of Section~\ref{sec.concrete} suggests, we aim to work here in a less abstract way. So, first of all, we fix a representation of our underlying algebra as $C(X)=\bigoplus_{i=1}^aM_{n_i}(\C)$.

We defined a quantum graph over $X$ to be given by a projection $\tilde A\in C(X)\otimes C(X)^{\rm op}$. As we mentioned several times in this paper already, having a concrete realization of $C(X)$ by matrices, it is convenient to identify $C(X)^{\rm op}$ with $C(X)$ through transposition. So, we have
$$\tilde A\in\left(\bigoplus_{i=1}^a M_{n_i}(\C)\right)\otimes\left(\bigoplus_{j=1}^a M_{n_j}(\C)\right)$$

Equivalently, such a projection is determined by a set of projections
$$\tilde A_{ij}\in M_{n_i}(\C)\otimes M_{n_j}(\C).$$

Those projections act on the vector space $\C^{n_i}\otimes\C^{n_j}$. We can identify the tensor product $\C^{n_i}\otimes\C^{n_j}$ with the vector space of linear operators $\C^{n_i}\to\C^{n_j}$. Every projection acting on some vector space is determined by its range. Consequently our quantum graph is equivalently described by a set of vector subspaces
$$V_{ij}\subset \C^{n_i}\otimes\C^{n_j}\simeq\Lin(\C^{n_i},\C^{n_j})$$

Finally, we can put all those vector spaces together and observe that a quantum graph is given by a single operator space
$$V\subset \Lin(\C^n)\simeq M_n(\C)$$
which can be decomposed as $V=\bigoplus V_{ij}$ with $V_{ij}\subset\Lin(\C^{n_i},\C^{n_j})$. One can check that this decomposition condition can also be written as $C(X)'VC(X)'\subset V$, where $C(X)'$ is the commutant of $C(X)$.

The structure of a general quantum graph is illustrated in Figure~\ref{fig.qgraph}.

In the proposition below, we summarize the characterization that we just derived. As was already mentioned, it is actually the original definition of a quantum graph by Weaver from \cite{Wea12} (note that our \emph{quantum graphs} are called \emph{quantum relations} in \cite{Wea12,Wea21}; in contrast, Weaver's quantum graphs are assumed to be undirected and have a loop at every vertex).

\begin{prop}[Equivalent definition of a quantum graph]
Let $X$ be a finite quantum set. A quantum graph $(X,\tilde A)$ is equivalently described by an operator space $V\subset\Lin(\C^n)$ such that $C(X)'VC(X)'\subset V$. The quantum graph
\begin{enumerate}
\item is undirected if and only if $V=V^\dag$,
\item has no loops if and only if $\Tr_i\xi=0$ for every $\xi\in V$ and every $i=1,\dots,a$, where $\Tr_i$ is the partial trace related to the subspace $\C^{n_i}\subset\C^n$,
\item[(2*)] has a loop at every vertex if and only if $C(X)'\subset V$.
\end{enumerate}
\end{prop}
\begin{proof}
The correspondence between quantum graphs and operator spaces was already described above. See \cite[Proposition~2.23]{Wea12} or \cite[Proposition~VII.5]{MRV18} for a more detailed proof.

We are going to present a detailed proof for the case $X=M_n$ in the following section.
\end{proof}

\subsection{Case $X=M_n$}
In this section, we will look at quantum graphs over the quantum space $M_n$. We will work in the standard orthogonal basis $(e_{ij})$. Recall from Example~\ref{E.MnAdj} that the adjacency matrix $A\colon l^2(X_n)\to l^2(X_n)$ of a graph $(M_n,\tilde A)$ is given by $A^{ij}_{kl}=n\tilde A^{ijkl}$. Here, we again use the ``transposed'' version of $\tilde A\in C(X)\otimes C(X)=M_n(\C)\otimes M_n(\C)$.

The edge projection $\tilde A\in M_n(\C)\otimes M_n(\C)$ acts on the vector space $\C^n\otimes\C^n$. Similarly to the ``rotation'' idea for adjacency matrices, it will be also convenient to rotate vectors $\tilde\xi\in\C^n\otimes\C^n$:

\begin{nota}
For a given vector $\tilde\xi\in\C^n\otimes\C^n$, we define a linear operator $\xi\colon\C^n\to\C^n$ by $\xi^i_j=\sqrt{n}\,\tilde\xi^{ij}$ and vice versa. We will write $\xi\in\glin_n$ instead of $\xi\in M_n(\C)$ in order to distinguish this kind of matrices from the elements of $C(X)$.

We will denote by $\tilde\xi^*$ the dual functional on $\C^n\otimes\C^n$ with entries $[\tilde\xi^*]_{ij}=\overline{\tilde\xi^{ij}}$. We will denote by $\xi^*$ the complex conjugate of $\xi$, i.e.\ $[\xi^*]^i_j=\overline{\xi^i_j}$. Finally, we denote by $\xi^\dag$ the conjugate transpose of $\xi$, i.e.\ $[\xi^\dag]^i_j=\overline{\xi^j_i}$.
\end{nota}

\begin{defn}
A \emph{quantum edge} on $M_n$ is a minimal projection $\tilde P\in M_n(\C)\otimes M_n(\C)$. We say that
\begin{enumerate}
\item it is \emph{symmetric} if $(M_n,\tilde P)$ is undirected,
\item it contains \emph{no loop} if $(M_n,\tilde P)$ has no loop, i.e.\ $\tilde P\tilde I=0$,
\end{enumerate}
\end{defn}

\begin{rem}
Every quantum edge on $M_n$ contains $n^2$ edges in the sense that
$$(\eta^\dag\otimes\eta^\dag)\tilde P=n^2\Tr\tilde P=n^2.$$
\end{rem}

\begin{prop}\label{P.edge}
The quantum edges on $M_n$ are in a one-to-one correspondence with elements $\xi\in\glin_n$ up to normalization. In addition, the edge
\begin{itemize}
\item contains no loop if and only if $\Tr\xi=0$, i.e. $\xi\in\slin_n$,
\item is symmetric if and only if (up to normalization) $\xi=\xi^\dag$, i.e.\ $\im\xi\in\un_n$,
\item so, it is both symmetric and contains no loop if and only if (up to normalization) $\xi\in\su_n$.
\end{itemize}
The edge projection is given by $\tilde P=\frac{1}{\tilde\xi^*\tilde\xi}\tilde\xi\tilde\xi^*$. Its adjacency matrix is then given by $P=\frac{n}{\Tr(\xi^\dag\xi)}\xi\otimes\xi^*$ (identifying $l^2(X)\simeq\C^n\otimes\C^n$).
\end{prop}
\begin{proof}
A quantum edge on $M_n$ is by definition given by a minimal projection $\tilde P$ in $M_n(\C)\otimes M_n(\C)$. Such a projection is determined by a vector $\tilde\xi\in\C^n\otimes\C^n$ by $\tilde P=\frac{1}{\|\tilde\xi\|^2}\tilde\xi\tilde\xi^*$, where $\|\tilde\xi\|^2=\sum_{ij}|\tilde\xi^{ij}|^2=\tilde\xi^*\tilde\xi$. On the other hand, this vector is defined by the projection $\tilde P$ uniquely up to normalization.

Such bivectors $\tilde\xi$ are of course in a one-to-one correspondence with elements $\xi\in\glin_n$. The norm of $\tilde\xi$ is equivalently computed as $\|\tilde\xi\|=\frac{1}{n}\Tr(\xi^\dag\xi)$. Let us assume for simplicity $\|\tilde\xi\|=1$.

On the other hand, $\tilde P$ is alternatively described by the rotated version $P\colon l^2(X_n)\to l^2(X_n)$. Let us compute its elements explicitly:
$$P^{ij}_{kl}=n\tilde P^{ijkl}=n\,\tilde\xi^{ik}\overline{\tilde\xi^{jl}}=\xi^i_k\overline{\xi^j_l}=[\xi\otimes\xi^*]^{ij}_{kl}.$$

It remains to reformulate the condition on containing no loop and being symmetric in terms of $\xi$. The quantum edge contains no loop if and only if $0=\tilde I\cdot\tilde P=\tilde I\tilde\xi\tilde\xi^*$, which holds if and only if $\tilde I\tilde\xi=0$, which means
$$0=[\tilde I\tilde\xi]^{ij}=\sum_{k,l}\tilde I^{ikjl}\tilde\xi^{kl}=n^{-3/2}\sum_{k,l}\delta_{ij}\delta_{kl}\xi^k_l=n^{-3/2}\delta_{ij}\Tr\xi.$$

The quantum edge is symmetric if and only if $P=P^\dag$, which holds obviously if and only if $\xi=\xi^\dag$ up to normalization.
\end{proof}

\begin{rem}
The projection $\tilde I$ corresponding to the empty graph with a loop at every vertex is minimal. Indeed, its rotated version $I$ can be written as $\iota\otimes\iota$, where $\iota\in\glin_n$ is the identity matrix, so by Proposition~\ref{P.edge}, $\tilde I$ is a quantum edge.

This means that there is a unique \emph{pure loop} on $M_n$ -- namely the quantum edge given by $\tilde I$. Nevertheless, one might still construct quantum edges $\tilde P$ on $M_n$, which are \emph{partial loops}. That is, $\tilde P$ as a projection neither equals $\tilde I$ nor it is orthogonal to~$\tilde I$.
\end{rem}

\begin{thm}\label{T}
Consider $m\in\{0,1,\dots,n^2\}$.
Quantum graphs on $M_n$ with $m$ quantum edges are in a one-to-one correspondence with $m$-dimensional subspaces $V\subset\glin_n$. The adjacency matrix is given by $A_V=\sum_{s=1}^m\xi_s\otimes\xi_s^*$, where $(\xi_1,\dots,\xi_m)$ is some orthonormal basis of $V$.

The quantum graph has no loops if and only if $V\subset\slin_n$. It is undirected if and only if $V=V^\dag$.

Consider two such graphs corresponding to $V,W\subset\glin_n$. They satisfy $\tilde A_V\le\tilde A_W$ if and only if $V\subset W$. They satisfy $(M_n,\tilde A_V)\simeq(M_n,\tilde A_W)$ if and only if $W=UVU^\dag$ for $U\in SU_n$.
\end{thm}
\begin{proof}
Most of the theorem follows already from Proposition~\ref{P.edge}. By a graph with $m$ quantum edges, we mean that the associated projection $\tilde A_V$ has dimension $m$, that is, it is projecting on an $m$-dimensional vector space $\tilde V\subset \C^n\otimes\C^n$. We denote $V=\{\xi\mid\tilde\xi\in\tilde V\}\subset\glin_n$. If $(\xi_1,\dots,\xi_m)$ is a basis of $V$, so $(\tilde\xi_1,\dots,\tilde\xi_m)$ is a basis of $\tilde V$, then obviously $\tilde A_V=\sum_i\tilde P_s$, where $\tilde P_s=\tilde\xi_s\tilde\xi_s^*$ are the corresponding one-dimensional projections. As we showed in the proof of Proposition~\ref{P.edge}, this means that $P_s=\xi_s\otimes\xi_s^*$, so $A_V=\sum_i\xi_s\otimes\xi_s^*$.

The graph $(M_n,\tilde A_V)$ has no loops, i.e. $A_V\bullet I=0$ if and only if $P_s\bullet I=0$ for every $i$. In Proposition~\ref{P.edge}, we showed that this is equivalent to $\Tr\xi_s=0$. The graph is undirected if and only if $A_V=A_V^\dag=\sum\xi_s^\dag\otimes\xi_s^{*\dag}$, which is equivalent to saying that $(\xi_1^\dag,\dots,\xi_m^\dag)$ is a basis of $V$, so $V=V^\dag$.

The fact that $\tilde A_V\le\tilde A_W$ if and only if $V\subset W$ is obvious.

Finally, we have $(M_n,\tilde A_V)\simeq(M_n,\tilde A_W)$ if and only if there is an automorphism $M_n\to M_n$ with underlying $*$-automorphism $\phi\colon M_n(\C)\to M_n(\C)$ such that $A_W=\phi\circ A_V\circ\phi^{-1}$. In the proof of Proposition~\ref{P.Aut}, we mentioned that any $*$-automorphism of $M_n(\C)$ is given by conjugating with a unitary (or equivalently a special unitary) matrix. Expressing this in entries, the automorphism acts on $x\in M_n(\C)$ by $[\phi(x)]^{ij}=\sum_{k,l}U^i_kx^{kl}\overline{U^j_l}$. In other words, we can say that the entries of $\phi$ are $\phi^{ij}_{kl}=U^i_k\overline{U^j_l}$. Obviously, $[\phi^{-1}]^{ij}_{kl}=\overline{U^k_i}U^l_j$ So, taking $A_V=\sum_a \xi_a\otimes\xi_a^*$, we get
\begin{align*}
[\phi\circ A_V\circ\phi^{-1}]^{ij}_{kl}&=\sum_{a,b,c,d}U^i_a\overline{U^j_b}[A_V]^{ab}_{cd}\overline{U_c^k}U_d^l=\sum_{a,b,c,d}\sum_s U^i_a[\xi_s]^a_c\overline{U_c^k}\,\overline{U^j_b[\xi_s]^b_d\overline{U^l_d}}\\
&=\sum_s[U\xi_s U^\dag]^i_k[(U\xi_sU^\dag)^*]^j_l.
\end{align*}
That is, $(M_n,\tilde A_V)\simeq(M_n,\tilde A_W)$ if and only if $W$ is spanned by $U\xi_s U^\dag$, so $W=UVU^\dag$.
\end{proof}

\begin{lem}
Any vector subspace $V\subset\glin_n$ with $V=V^\dag$ has a self-adjoint basis $(\xi_s)$, $\xi_s=\xi_s^\dag$.
\end{lem}
\begin{proof}
We construct the basis by induction. Suppose we already have a basis $(\xi_s)$ of a subspace $V_1\subset V$ with $\xi_s=\xi_s^\dag$. This in particular means that $V_1=V_1^\dag$. Denote by $V_2$ some vector space complement of $V_1$ to $V$. Take any element $\xi\in V_2$ and construct its real and imaginary part $(\xi+\xi^\dag)/2$ and $\im(\xi^\dag-\xi)/2$, which are both self-adjoint and both elements of $V_2$. It may or may not happen that those are linearly independent, but definitely at least one of these is non-zero. So, include both (if they are linearly independent) or one of them (if they are not) to our basis $(\xi_s)$ and repeat.
\end{proof}

\begin{rem}
As a consequence of the lemma, any undirected graph on $M_n$ can be thought of as being made from symmetric quantum edges only.
\end{rem}

\begin{rem}\label{R.real}
A complex vector space with $V=V^\dag\subset\glin_n$ uniquely determines a real vector space spanned by the basis $(\xi_s)$ of self adjoint elements $\xi_s=\xi_s^\dag$. Indeed, it is easy to see that any other self-adjoint element $\xi=\xi^\dag$ must be a real linear combination of $(\xi_s)$. Equivalently, $V$ is uniquely described by a real vector subspace of the Lie algebra $\un_n$.
\end{rem}

\subsection{Special case $n=2$}
\label{secc.smallclass}

In this section, we look at the special case $n=2$. We provide an explicit classification of simple graphs over $M_2$. Note that this result was obtained independently by Matsuda in \cite{Mat21}. The result of Matsuda is actually more general since he also classifies quantum graphs over $M_2$ with respect to non-tracial $\delta$-forms. We also mention the classification for undirected quantum graphs dropping the condition of {\em no loops}.

\begin{thm}\label{T2}
A simple graph on $M_2$ is determined up to isomorphism only by the number of quantum edges $m\in\{0,1,2,3\}$.
\end{thm}
\begin{proof}
Note first that we assume that the graph is simple, so there is no loop, so there can indeed be at most $2^2-1=3$ quantum edges.

By Theorem~\ref{T}, a simple graph on $M_2$ with $m$ quantum edges is determined by an $m$-dimensional subspace $V\subset\slin_2$ with $V=V^\dag$. By Remark~\ref{R.real}, we can consider $V$ to be a real vector space. We can equivalently also write $\im V\subset\su_2$. There is the famous basis of $\su_2$ by Pauli matrices. Expressing everything in this basis, we can also think of $V$ as a subspace $V\subset\R^3$.

By Theorem~\ref{T}, two such vector spaces $V,W$ describe isomorphic graphs if and only if they are conjugated by a matrix $U\in SU_2$. It is well known that the conjugation action of $SU_2$ on $\su_2$ is realized in the basis of Pauli matrices as rotations. In fact, $PU_2\simeq SO_3$.

So, two quantum graphs corresponding to subspaces $V,W\subset\R^3$ are isomorphic if and only if $V$ is a rotation of $W$. But this obviously holds if and only if $\dim V=\dim W$.
\end{proof}

\begin{rem}
As mentioned in \cite{Mat21}, the graphs are actually $m$-regular.
\end{rem}

\begin{ex}\label{E.2}
Let us compute explicitly the adjacency matrices corresponding to the graphs from Theorem~\ref{T2}.

For $m=0$, we of course have the empty graph, so $A=0$. Likewise, for $m=3$, we have the full graph without loops, so $A=J-I$. The only interesting cases are therefore $m=1,2$.

Recall the explicit form of the Pauli matrices
$$
\sigma_1=
\begin{pmatrix}
0&1\\
1&0
\end{pmatrix}
,\qquad
\sigma_2=
\begin{pmatrix}
0&-\im\\
\im&0
\end{pmatrix}
,\qquad
\sigma_3=
\begin{pmatrix}
1&0\\
0&-1
\end{pmatrix}.
$$

Note that they are correctly normalized, that is, $\Tr(\sigma_i^\dag\sigma_i)=2=n$. We can also write them as orthonormal vectors:
$$
\tilde\sigma_1=
\frac{1}{\sqrt{2}}
\begin{pmatrix}
0\\1\\1\\0
\end{pmatrix}
,\qquad
\tilde\sigma_2=
\frac{1}{\sqrt{2}}
\begin{pmatrix}
0\\-\im\\\im\\0
\end{pmatrix}
,\qquad
\tilde\sigma_3=
\frac{1}{\sqrt{2}}
\begin{pmatrix}
1\\0\\0\\-1
\end{pmatrix}.
$$

For $m=1$, we can choose $V=\spanlin\{\sigma_1\}$, so the adjacency matrix is given by
$$A=P_1:=\sigma_1\otimes\sigma_1^*=
\begin{pmatrix}
0&0&0&1\\
0&0&1&0\\
0&1&0&0\\
1&0&0&0
\end{pmatrix}.
$$

By this, we mean that $A\colon l^2(X_2)\to l^2(X_2)$ maps
\begin{align*}
 \begin{pmatrix}1&0\\0&0\end{pmatrix}\mapsto\begin{pmatrix}0&0\\0&1\end{pmatrix},\qquad
&\begin{pmatrix}0&1\\0&0\end{pmatrix}\mapsto\begin{pmatrix}0&0\\1&0\end{pmatrix},\cr
 \begin{pmatrix}0&0\\1&0\end{pmatrix}\mapsto\begin{pmatrix}0&1\\0&0\end{pmatrix},\qquad
&\begin{pmatrix}0&0\\0&1\end{pmatrix}\mapsto\begin{pmatrix}1&0\\0&0\end{pmatrix}.
\end{align*}

Alternatively, we can compute the edge projection as
$$
\tilde A=\tilde P_1=\tilde\sigma_1\tilde\sigma_1^*=
\frac{1}{2}
\begin{pmatrix}
\begin{pmatrix}0&0\\0&1\end{pmatrix}&
\begin{pmatrix}0&0\\1&0\end{pmatrix}\\
\begin{pmatrix}0&1\\0&0\end{pmatrix}&
\begin{pmatrix}1&0\\0&0\end{pmatrix}
\end{pmatrix}.
$$

We can also take any other traceless self-adjoint matrix to generate a simple quantum graph with one edge. For instance, choosing $V=\spanlin\{\sigma_2\}$, resp. $V=\spanlin\{\sigma_3\}$, we obtain $A=P_2$ or $A=P_3$, where
\begin{align*}
P_2&=\sigma_2\otimes\sigma_2^*=
\begin{pmatrix}
0&0&0&1\\
0&0&-1&0\\
0&-1&0&0\\
1&0&0&0
\end{pmatrix}
&\leftrightarrow&&
\tilde P_2&=\tilde\sigma_2\tilde\sigma_2^*
=\frac{1}{2}
\begin{pmatrix}
0&0&0&0\\
0&1&-1&0\\
0&-1&1&0\\
0&0&0&0
\end{pmatrix}
\\
P_3&=\sigma_3\otimes\sigma_3^*=
\begin{pmatrix}
1&0&0&0\\
0&-1&0&0\\
0&0&-1&0\\
0&0&0&1
\end{pmatrix}
&\leftrightarrow&&
\tilde P_3&=\tilde\sigma_3\tilde\sigma_3^*
=\frac{1}{2}
\begin{pmatrix}
1&0&0&-1\\
0&0&0&0\\
0&0&0&0\\
-1&0&0&1
\end{pmatrix}.
\end{align*}

A quantum graph with two edges can be constructed simply by summing the adjacency matrices of two graphs with one edge (the vectors corresponding to the two edges should be mutually orthogonal). For example, we can take $A=P_1+P_2$. Actually, in Example~\ref{E.square}, we already presented such an example. In that case, we chose $A=P_1+P_3$.

Finally, as we already mentioned, taking three edges, we obtain the full graph without loops with adjacency matrix $A=P_1+P_2+P_3=J-I$.
\end{ex}

We can also generalize the result of Theorem~\ref{T2} by dropping the assumption that the graph has no loops. Of course, for any graph constructed in Theorem~\ref{T2}, we can define a new graph by adding the unique pure loop $\tilde I$. That is, we define $A'=A+I$. But apart from those trivial examples, there are infinitely many other graphs that contain the loop only partially.

\begin{prop}
\label{P.partialex}
For any $m=1,2,3$, there is an infinite family of mutually non-isomorphic undirected quantum graphs containing $m$ quantum edges. They can be chosen as follows:
\begin{itemize}
\item For $m=1$, we define $A(t)=P_3(t)$, $t\in[0,\pi/2]$.
\item For $m=2$, we define $A(t)=P_3(t)+P_2$, $t\in[0,\pi/2]$.
\item For $m=3$, we define $A(t)=P_3(t)+P_2+P_1$, $t\in[0,\pi/2]$.
\end{itemize}
Here, we denote $P_i=\sigma_i\otimes\sigma_i^*$ as in Example~\ref{E.2}. In addition, we define $P_3(t)=\sigma_3(t)\otimes\sigma_3(t)^*$, where $\sigma_3(t)=\sigma_3\sin t+\iota\cos t$.

Every other undirected quantum graph is isomorphic to one of these families (or it equals to the empty graph or the full graph).
\end{prop}
To be more concrete, we can write down explicitly
$$P_3(t)=
\begin{pmatrix}
1+\sin(2t)&0&0&0\\
0&\cos(2t)&0&0\\
0&0&\cos(2t)&0\\
0&0&0&1-\sin(2t)
\end{pmatrix}
$$
For $t=0$, we get $I$, i.e.\ the pure loop. For $t=\pi/2$, we get $P_3$, i.e.\ no loop. For $t\in(0,\pi/2)$, we get something in between.
\begin{proof}
We just have to modify the proof of Theorem~\ref{T2}. Again, Theorem~\ref{T} tells us that all quantum graphs with $m$ edges are in a one-to-one correspondence to $m$-dimensional subspaces $V\subset\R^4$. 
\end{proof}

\section{Quantum symmetries and isomorphisms}
\label{sec.qsym}

\subsection{Compact matrix quantum groups}
\label{secc.qgdef}
A~\emph{compact matrix quantum group} is a pair $G=(A,u)$, where $A$ is a~$*$-algebra and $u=(u^i_j)\in M_N(A)$ is a matrix with values in $A$ such that
\begin{enumerate}
\item the elements $u^i_j$ $i,j=1,\dots, N$ generate $A$,
\item the matrices $u$ and $u^t$ ($u$ transposed) are similar to unitary matrices,
\item the map $\Delta\colon A\to A\otimes A$ defined as $\Delta(u^i_j):=\sum_{k=1}^N u^i_k\otimes u^k_j$ extends to a~$*$-homomorphism.
\end{enumerate}

Compact matrix quantum groups introduced by Woronowicz \cite{Wor87} are generalizations of compact matrix groups in the following sense. For a~matrix group $G\subseteq M_N(\C)$, we define $u^i_j\colon G\to\C$ to be the coordinate functions $u^i_j(g):=g^i_j$. Then we define the \emph{coordinate algebra} $A:=\Olg(G)$ to be the algebra generated by $u^i_j$. The pair $(A,u)$ then forms a compact matrix quantum group. The so-called \emph{comultiplication} $\Delta\colon \Olg(G)\to \Olg(G)\otimes \Olg(G)$ dualizes matrix multiplication on $G$: $\Delta(f)(g,h)=f(gh)$ for $f\in \Olg(G)$ and $g,h\in G$.

Therefore, for a~general compact matrix quantum group $G=(A,u)$, the algebra~$A$ should be seen as the algebra of non-commutative functions defined on some non-commutative compact underlying space. For this reason, we often denote $A=\Olg(G)$ even if $A$ is not commutative. The matrix $u$ is called the \emph{fundamental representation} of~$G$.

Actually, $\Olg(G)$ also has the structure of a~Hopf $*$-algebra. That is, there is a \emph{counit} $\epsilon\colon\Olg(G)\to\C$ mapping $u^i_j\mapsto\delta_{ij}$ (classical meaning is evaluating a function at identity, i.e.\ $\epsilon(f)=f(e)$) and an \emph{antipode} $S\colon\Olg(G)\to\Olg(G)$ mapping $u^i_j\mapsto [u^{-1}]^i_j$ (classical meaning is evaluating a function at the group inverse, i.e.\ $[S(f)](g)=f(g^{-1})$). In addition, we can also define the C*-algebra $C(G)$ as the universal C*-completion of $A$, which can be interpreted as the algebra of continuous functions of $G$.

A compact matrix quantum group $H=(\Olg(H),v)$ is a~\emph{quantum subgroup} of $G=(\Olg(G),u)$, denoted as $H\subseteq G$, if $u$ and $v$ have the same size and there is a~surjective $*$-homomorphism $\phi\colon \Olg(G)\to \Olg(H)$ sending $u^i_j\mapsto v^i_j$. We say that $G$ and $H$ are \emph{equal} if there exists such a $*$-isomorphism (i.e. if $G\subset H$ and $H\subset G$). We say that $G$ and $H$ are \emph{isomorphic} if there exists a $*$-isomorphism $\phi\colon \Olg(G)\to \Olg(H)$ such that $(\phi\otimes\phi)\circ\Delta_G=\Delta_H\circ\phi$. We will often use the notation $H\subseteq G$ also if $H$ is isomorphic to a subgroup of $G$.

One of the most important examples is the quantum generalization of the orthogonal group -- the \emph{free orthogonal quantum group} defined by Wang in \cite{Wan95free} through the universal $*$-algebra
$$\Olg(O_N^+):=\staralg(u^i_j,\;i,j=1,\dots,N\mid u^i_j=u^{i*}_j,uu^t=u^tu=\id_{\C^N}).$$
Note that this example was then further generalized by Banica~\cite{Ban96} into the \emph{universal free orthogonal quantum group} $O^+(F)$, where $F\in M_N(\C)$ such that $F\bar F=\pm\id_{\C^N}$ and
$$\Olg(O^+(F)):=\staralg(u^i_j,\;i,j=1,\dots,N\mid\text{$u$ is unitary, }u=F\bar uF^{-1}),$$
where $[\bar u]^i_j=u^{i*}_j$.

\subsection{Representation categories and Tannaka--Krein reconstruction}
\label{secc.Rep}

For a~compact matrix quantum group $G=(\Olg(G),u)$, we say that $v\in M_n(\Olg(G))$ is a~representation of $G$ if $\Delta(v^i_j)=\sum_{k}v^i_k\otimes v^k_j$, where $\Delta$ is the comultiplication. The representation $v$ is called \emph{unitary} if it is unitary as a~matrix, i.e. $\sum_k v^i_kv^{j*}_k=\sum_k v^{k*}_iv^k_j=\delta_{ij}$. In particular, an element $a\in \Olg(G)$ is a one-dimensional representation if $\Delta(a)=a\otimes a$. Another example of a quantum group representation is the fundamental representation $u$.

We say that a representation $v$ of $G$ is \emph{non-degenerate} if $v$ is invertible as a matrix. Equivalently, $v$ is non-degenerate if $\epsilon(v^i_j)=\delta_{ij}$. (In the classical group theory, we typically consider only non-degenerate representations; therefore, it is sometimes put as a requirement also in the quantum case.) It is \emph{faithful} if $\Olg(G)$ is generated by the entries of~$v$. The meaning of this notion is the same as with classical groups: Given a non-degenerate faithful representation $v$, the pair $G'=(\Olg(G),v)$ is also a compact matrix quantum group, which is isomorphic to the original $G$.


For two representations $v\in M_n(\Olg(G))$, $w\in M_m(\Olg(G))$ of $G$ we define the space of \emph{intertwiners}
$$\Mor(v,w)=\{T\colon \C^n\to\C^m\mid Tv=wT\}.$$
The set of all representations of a given quantum group together with those intertwiner spaces form a rigid monoidal $\dag$-category, which will be denoted by $\Rep G$.

Nevertheless, since we are working with compact \emph{matrix} quantum groups, it is more convenient to restrict our attention only to certain representations related to the fundamental representation. If we work with orthogonal quantum groups $G\subset O_n^+$ (or $G\subset O^+(F)$ in general), then it is enough to focus on the tensor powers $u^{\otimes k}$ since the entries of those representations already linearly span the whole $\Olg(G)$.

Considering $G=(\Olg(G),u)\subset O^+(F)$, $F\in M_N(\C)$, we define
$$\RCat_G(k,l):=\Mor(u^{\otimes k},u^{\otimes l})=\{T\colon (\C^N)^{\otimes k}\to(\C^N)^{\otimes l}\mid Tu^{\otimes k}=u^{\otimes l}T\}.$$

The collection of such linear spaces forms a rigid monoidal $\dag$-category with the monoid of objects being the natural numbers with zero $\N_0$.

Conversely, we can reconstruct any compact matrix quantum group from its representation category \cite{Wor88,Mal18}.

\begin{thm}[Woronowicz--Tannaka--Krein]
\label{T.TK}
Let $\RCat$ be a rigid monoidal $\dag$-category with $\N_0$ being the set of objects and $\RCat(k,l)\subset\Lin((\C^N)^{\otimes k},(\C^N)^{\otimes l})$. Then there exists a unique orthogonal compact matrix quantum group $G$ such that $\RCat=\RCat_G$. We have $G\subset O^+(F)$ with $F^j_i=R^{ij}$, where $R$ is the duality morphism of $\RCat$.
\end{thm}

We can write down the associated quantum group very concretely. The relations satisfied in the algebra $\Olg(G)$ will be exactly the intertwining relations:
$$\Olg(G)=\staralg(u^i_j,\;i,j=1,\dots,N\mid u=F\bar uF^{-1},\; Tu^{\otimes k}=u^{\otimes l}T\;\forall T\in\RCat(k,l)).$$

We say that $S$ is a generating set for a representation category $\RCat$ if $\RCat$ is the smallest monoidal $\dag$-category satisfying the assumptions of Theorem~\ref{T.TK} that contains $S$. We use the notation $\RCat=\langle S\rangle_N$ (it is important to specify the dimension $N$ of the vector space $V=\C^N$ associated to the object 1). If we know such a generating set, it is enough to use the generators for our relations:
$$\Olg(G)=\staralg(u^i_j,\;i,j=1,\dots,N\mid u=F\bar uF^{-1},\; Tu^{\otimes k}=u^{\otimes l}T\;\forall T\in S(k,l)).$$

\subsection{Quantum symmetries}
\label{secc.Qsym}

We are going to define the quantum automorphism group $G$ of a given quantum space $X$ to be the universal quantum group acting on $X$. First, we need to explain, what such an action means.

\begin{defn}
Let $X$ be a quantum space and $G$ a compact quantum group. An \emph{action} of $G$ on $X$ is described by a \emph{coaction} of $\Olg(G)$ on $C(X)$, that is, a $*$-homomorphism $\alpha\colon C(X)\to C(X)\otimes\Olg(G)$, which is
\begin{itemize}
\item \emph{coassociative}: $(\alpha\otimes\id)\alpha(x)=(\id\otimes\Delta)\alpha(x)$,
\item \emph{counital}:      $(\id\otimes\epsilon)\alpha(x)=x$,
\item $\eta^\dag$-preserving:  $(\eta^\dag\otimes\id)\alpha(x)=\eta^\dag(x)1_{\Olg(G)}$.
\end{itemize}
\end{defn}

These conditions can be conveniently characterized in the language of representations and intertwiners. The following two lemmata reformulate and expand \cite[Lemma 1.1, 1.2]{Ban99}.

\begin{nota}\label{N.ONB}
For the rest of this subsection, we fix an orthonormal basis $(e_i)_{i=1}^N$ of $l^2(X)$. As in Lemma~\ref{L.RF}, we denote by $F$ the matrix of the $*$-operation in $C(X)$, that is, $e_i^*=:\sum_j e_j\bar F_i^j$. Recall that $F_i^j=R^{ij}$. Consider some linear map $\alpha\colon C(X)\to C(X)\otimes \Olg(G)$ and denote by $u_i^j$ its matrix entries with respect to the basis $(e_i)$. That is, $\alpha(e_i)=\sum_j e_j\otimes u^j_i$. 
\end{nota}

\begin{lem}
Considering Notation~\ref{N.ONB}, $\alpha$ is coassociative if and only if $u$ is a representation of $G$. In addition, it is counital if and only if the representation $u$ is non-degenerate.
\end{lem}
\begin{proof}
We just expand both sides of the coassociativity condition.
\begin{align*}
[(\alpha\otimes\id)\alpha](e_i)&=\sum_k e_k\otimes u^k_j\otimes u^j_i,\\
[(\id\otimes\Delta)\alpha](e_i)&=\sum_k e_k\otimes\Delta(u^k_i).
\end{align*}
So, $(\alpha\otimes\id)\alpha(e_i)=(\id\otimes\Delta)\alpha(e_i)$ if and only if $\Delta(u^i_j)=\sum_ku^i_k\otimes u^k_j$.

Similarly, rewriting the condition $(\id\otimes\epsilon)\alpha(e_i)=e_i$ using the entries $u^i_j$ we obtain $\epsilon(u^i_j)=\delta_{ij}$, which is equivalent to $u$ being non-degenerate.
\end{proof}

\begin{lem}
\label{L.alphaMor}
Let $X$ be a quantum space and $G$ a compact quantum group. Consider Notation~\ref{N.ONB} and suppose that $u$ is a non-degenerate representation. Then the following holds:
\begin{enumerate}
\item $\alpha$ is involutive if and only if $F\in\Mor(\bar u,u)$,
\item $\alpha$ is multiplicative if and only if $m\in\Mor(u\otimes u,u)$,
\item $\alpha$ is unital if and only if $\eta\in\Mor(1,u)$,
\item $\alpha$ is $\eta^\dag$-preserving if and only if $\eta^\dag\in\Mor(u,1)$,
\item $\alpha$ preserves the bilinear form $(\cdot,\cdot)$ if and only if $R^\dag\in\Mor(u\otimes u,1)$,
\item $\alpha$ preserves the inner product $\langle\cdot,\cdot\rangle$ if and only if $u$ is unitary.
\end{enumerate}
Supposing that (2) and (6) hold, then $(1)\Leftrightarrow (3)\Leftrightarrow (4)\Leftrightarrow(5)$.
\end{lem}
\begin{proof}
All items are easily proven by rewriting the condition in matrix entries. The first one:
\begin{align*}
\alpha(e_i^*)&=\sum_kF^k_i\alpha(e_k^*)=\sum_{j,k}F^k_ie_j\otimes u^j_k\\
\alpha(e_i)^*&=\sum_ke_k^*\otimes u^{k*}_i=\sum_{j,k}F^j_ke_j\otimes u^{k*}_i
\end{align*}
That is, $\alpha(e_i^*)=\alpha(e_i)^*$ if and only if $\sum_k u^j_kF^k_i=\sum_k F^j_ku^{k*}_i$, that is, $uF=F\bar u$, that is, $F\in\Mor(u,\bar u)$.

Similarly, we can prove all the other items. Actually, one just needs to realize that $(u^j_i)$ is just the matrix expression for the coaction $\alpha$. Then it becomes almost obvious that (2) the multiplicativity $\alpha(x)\alpha(y)=\alpha(xy)$ can be equivalently expressed as $m(u\otimes u)=um$; (3) the unitality $\alpha(\eta)=\eta\otimes 1$ is equivalent to $u\eta=\eta$; (4) the $\eta^\dag$-preserving property is equivalent to $\eta^\dag u=\eta^\dag$; (5) preserving the bilinear form is equivalent to $R^\dag(u\otimes u)=R^\dag$; (6) preserving the inner product means nothing but $\delta_{ij}=\langle e_i,e_j\rangle=\sum\langle e_k,e_l\rangle\otimes u^{k*}_iu^l_j=\sum_ku^{k*}_iu^k_j$ (since we assume non-degeneracy, $u$ is invertible and hence $u^\dag u=\id$ already implies $uu^\dag=\id$).

Finally, assume that (2) and (6) hold. From (6), it follows that the representation category of $u$ must be closed under involution, so $(3)\Leftrightarrow(4)$. Since the bilinear form and the inner product differ only in taking the involution, we have ${(1)\Leftrightarrow(5)}$. Finally, it holds that $R^\dag=\eta^\dag m$ and, conversely, $\eta^\dag=\eta^\dag mm^\dag=R^\dag m^\dag$, so ${(4)\Leftrightarrow(5)}$.
\end{proof}


\begin{defn}\label{D.QAut0}
Let $X$ be a quantum space with a distinguished orthonormal basis $(e_i)_{i=1}^N$ of $C(X)$. The quantum symmetry group of $X$ is the compact matrix quantum group $S^+_X=(\Olg(S^+_X),u)$ with
$$\Olg(S^+_X)=\staralg(u^i_j,i,j=1,\dots,N\mid\text{$u$ is unitary, }uF=F\bar u,\;m(u\otimes u)=um).$$
\end{defn}

That is, we have constructed $S^+_X$ to be the unique compact matrix quantum group $G\subset O^+(\bar F)$, whose representation category is generated by $m$ and consequently satisfies all the conditions of Lemma~\ref{L.alphaMor}. Equivalently, we can say that $S^+_X$ is the ``universal'' quantum group acting on $X$.

\begin{ex}
The quantum symmetry groups of finite quantum spaces were studied in \cite{Wan98,Ban99}. For $X$ being the classical space of $N$ points, the quantum symmetries are described by the well-known \emph{free symmetric quantum group} $S_N^+$ introduced by Wang \cite{Wan98}. As for the quantum space corresponding to the matrix algebra $M_n(\C)$, the quantum symmetries are described by the projective free orthogonal quantum group $PO_n^+$ \cite[Corollary~4.1]{Ban99}. (The latter might be slightly confusing as we derived in Proposition~\ref{P.Aut} that $\Aut M_n=PU_n$ and we would expect this to be a quantum subgroup of $S_{M_n}^+$. Nevertheless, it actually holds that $PO_n^+=PU_n^+$ \cite[Theorem~4.1]{GProj}.)
\end{ex}

\begin{defn}
\label{D.QAut}
Let $X$ be a quantum graph. The quantum automorphism group of $X$ is the compact matrix quantum group $\Aut^+ X=(\Olg(\Aut^+ X),u)$ with
$$\Olg(\Aut^+X)=\staralg\left(u^i_j,i,j=1,\dots,N\middle|\begin{matrix}\text{$u$ is unitary, }uF=F\bar u,\\m(u\otimes u)=um,\;Au=uA\end{matrix}\right)$$
\end{defn}

That is, we added the adjacency matrix $A$ into the representation category of $\Aut^+X$.

\subsection{Quantum isomorphisms}

\begin{defn}
\label{D.qiso}
Let $X_1,X_2$ be weighted quantum graphs with adjacency matrices $A_1:=A_{X_1}$, $A_2:=A_{X_2}$. Let us also denote their quantum automorphism groups by $G_1:=\Aut^+X_1$, $G_2:=\Aut^+X_2$ and the corresponding fundamental representations by $u_1,u_2$. We say that $X_1$ and $X_2$ are \emph{quantum isomorphic} if there is a monoidal equivalence $\Rep G_1\to\Rep G_2$ mapping $u_1\mapsto u_2$ and $A_1\mapsto A_2$. That is, equivalently, a monoidal $\dag$-category isomorphism $\RCat_{G_1}\to\RCat_{G_2}$ mapping $A_1\mapsto A_2$
\end{defn}

There are many different equivalent (and also some non-equivalent) definitions of quantum isomorphisms. We took this one from \cite[Theorem~4.7]{BCE+20}. See also \cite{MRV19,LMR20,BCE+20,MR20} for other approaches.

More important than this definition is the following kind of a constructive converse:

\begin{thm}[{\cite[Theorem~4.11]{BCE+20}}]
\label{T.qiso}
Let $X$ be a weighted quantum graph, denote $G:=\Aut^+X$. Let $G'$ be some other compact matrix quantum group that is monoidally equivalent to $G$. Then there is a weighted quantum graph $X'$ with $\Aut^+X'=G'$, which is quantum isomorphic to $X$ through the given monoidal equivalence.
\end{thm}

The proof of this theorem is constructive. As also follows from Definition~\ref{D.qiso}, the monoidal equivalence has to map the fundamental representation $u$ of $G$ to the fundamental representation $u'$ of $G'$. We then define $C(X')$ to be the comodule space of $u'$ equipped with the multiplication $m'\in\Mor(u'\otimes u',u')$ given by the image of the multiplication $m\in\Mor(u\otimes u,u)$. The counit is also defined as the image of the counit $\eta^\dag\in\Mor(u,1)$ and finally the adjacency operator $A_{X'}$ is the image of the adjacency operator $A_X\in\Mor(u,u)$.

In the article \cite{BCE+20}, the proof is actually formulated for simple quantum graphs instead of our quantum graphs. But these additional conditions do not play any particular role in the proof and hence the statement can be proven exactly the same way. The additional conditions of simple quantum graphs can be handled separately in the following proposition.

\begin{prop}\label{P.qisomulti}
Quantum isomorphism of graphs preserves the following properties:
\begin{enumerate}
\item being a multigraph, i.e. $\mathop{\rm spec}(\tilde A)\subset\N_0$
\item being a graph, i.e. $A\bullet A=A$,
\item being undirected, i.e. $A=A^\dag$,
\item having no loops, i.e. $A\bullet I=m(A\otimes I)m^\dag=0$,
\item having a loop at every vertex, i.e. $A\bullet I=m(A\otimes I)m^\dag=I$,
\item the number of vertices $\#X=\dim C(X)=\eta^\dag\eta$,
\item the number of edges $\#E(X,\tilde A)=\eta^\dag A\eta$.
\end{enumerate}
\end{prop}
\begin{proof}
We just apply the monoidal equivalence to the characterizing equations. (For item (1), note that the monoidal equivalence induces an isomorphism of the endomorphism algebra $\Mor(u,u)$ not only with respect to the standard composition, but also with respect to the Schur product $\bullet$, so it must in particular preserve the properties of $A$ with respect to this product such as its spectrum.)
\end{proof}

\begin{rem}
It seems like quantum isomorphisms preserve every reasonable property. The truth is that it preserves every reasonable property formulated in terms of the representation category of $\Aut^+X$. But as we are going to see in the subsequent text, there are also many other reasonable properties, which are not preserved. In particular, the quantum automorphism group itself $\Aut^+X$ is not preserved (there are pairs of non-isomorphic quantum groups which are monoidally equivalent). Also the properties that rely on classical isomorphisms and homomorphisms are not preserved such as the classical automorphism group $\Aut X$ or the structure of quantum subgroups.
\end{rem}

\section{Twisting Cayley graphs of abelian groups}
\label{sec.twist}

In this section, we present a general construction for deforming Cayley graphs of an abelian groups using a 2-cocycle (equivalently a unitary bicharacter) on the group $\Gamma$. Our construction is actually a special case of 2-cocycle deformations introduced in \cite{MRV19}. However, in the case of Cayley graphs, the situation simplifies considerably.

\subsection{Cayley graphs of abelian groups}
\label{secc.Cayley}
Let us make a brief introduction to the theory of Cayley graph of abelian groups.

Let $\Gamma$ be a finite abelian group. We denote by $\Irr\Gamma\subset C(\Gamma)$ the set of all irreducible characters (that is, one-dimensional representations; since $\Gamma$ is abelian, all irreducible representations are in fact one-dimensional). Note that $\Irr\Gamma$ forms a basis of $C(\Gamma)$ and expressing a function in this basis is exactly the \emph{Fourier transform} on $\Gamma$.

Actually, the irreducible characters can be indexed by the group elements of $\Gamma$: Each element $\mu\in\Gamma$ defines an irreducible character $\tau_\mu$ such that $\tau_\mu\tau_\nu=\tau_{\mu+\nu}$, where + denotes the group operation. So, these characters provide actually a realization of the group $\Gamma$ and the Fourier transform provides an isomorphism $C(\Gamma)\simeq\C\Gamma$.

We denote by $\F$ the Fourier transformation matrix $\F^\alpha_\mu=\tau_\mu(\alpha)$. It holds that $[\F^{-1}]^\mu_\alpha=\frac{1}{N}\bar\F^\alpha_\mu=\frac{1}{N}\tau_\mu(-\alpha)$. If the group $\Gamma$ is realized as $\Gamma=\Z_{n_1}\times\cdots\times\Z_{n_m}$, then the characters can be realized as $\tau_\mu(\alpha)=\prod_{i=1}^m \omega_i^{\alpha_i\mu_i}$, where $\omega_i$ is some primitive $n_i$-th root of unity for every $i$. 

Consider a set $S\subset\Gamma$. The \emph{Cayley graph} of the group $\Gamma$ with respect to $S$ denoted by $\Cay(\Gamma,S)$ is a directed graph defined on the vertex set $\Gamma$ with an edge $(\alpha,\beta)$ for every pair of elements such that $\beta=\alpha+\theta$ for some $\theta\in S$. The set $S$ is usually assumed to be a generating set of $\Gamma$, which then implies that $\Cay(\Gamma,S)$ is connected. If $S$ is closed under the group inversion, then the Cayley graph is actually undirected (for every edge, one also has the opposite one). The adjacency matrix of $\Cay(\Gamma,S)$ is of the form
$$[A]^\beta_\alpha=
\begin{cases}
1&\text{if $\beta=\alpha+\theta$ for some $\theta\in S$,}\\
0&\text{otherwise.}
\end{cases}
$$

\begin{prop}[\cite{Lov75,Bab79}]
\label{P.lambda}
Let $\Gamma$ be a finite abelian group and $S\subset\Gamma$. Denote by $A$ the adjacency matrix associated to the Cayley graph $\Cay(\Gamma,S)$. Then $\Irr\Gamma$ forms the eigenbasis of $A$. Given $\mu\in\Gamma$, the eigenvalue corresponding to $\tau_\mu$ is given by
\begin{equation}
\label{eq.lambda}
\lambda_\mu=\sum_{\theta\in S}\tau_\mu(-\theta).
\end{equation}
\end{prop}
\begin{proof}
This is just a straightforward check. Take any $\chi\in\Irr\Gamma$. Then we have
\[[A\tau_\mu]^\alpha=\sum_{\beta\in\Gamma}A^\alpha_\beta\tau_\mu(\beta)=\sum_{\theta\in S}\tau_\mu(\alpha-\theta)=\sum_{\theta\in S}\tau_\mu(-\theta)\tau_\mu(\alpha)=\lambda_\mu\tau_\mu(\alpha).\qedhere\]
\end{proof}

\begin{ex}[Hypercube graph]\label{E.hypercube}
For any $n\in\N$, we define the \emph{hypercube graph} $Q_n$ to be the graph formed by vertices and edges of an $n$-dimensional hypercube. It can be interpreted as the Cayley graph of the group $\Gamma=\Z_2^n$ with respect to the generating set $\{\epsilon_1,\dots,\epsilon_n\}$, where $\epsilon_i=(0,\dots,0,1,0,\dots,0)$ is the standard generator of the $i$-th copy of $\Z_2$. We can use formula~\eqref{eq.lambda} to compute its spectrum as
$$\lambda_\mu=\sum_{i=1}^n\tau_\mu(\epsilon_i)=\sum_{i=1}^n(-1)^{\mu_i}=(n-2\deg\mu),$$
where $\deg\mu$ denotes the number of non-zero entries in $\mu$.
\end{ex}

In \cite{GroAbSym}, we showed how to use the knowledge about the spectrum and the eigenvectors to determine the quantum automorphism groups of cube-like graphs.

When working with ordinary (Cayley) graphs, one typically expresses everything in the standard basis $(\delta_\alpha)$ of $C(\Gamma)$. That is, the quantum automorphism group of the graph is the quantum subgroup of $S_N^+$ given by the relation $uA=Au$, where $u$ is the fundamental representation and $A$ is the ordinary adjacency matrix. The quantum automorphism group then acts by $\delta_i\mapsto\sum_j u_i^j\otimes\delta_j$.

However, in order to perform the twisting procedure in the subsequent text, it will be more convenient to express everything in the Fourier basis $(\tau_\mu)$. Here, the adjacency operator becomes a diagonal matrix $\hat A:=\F^{-1}A\F=\mathop{\rm diag}(\lambda_\mu)_{\mu\in\Gamma}$ as we just shown. The quantum automorphism group is then expressed using the fundamental representation $\hat u:=\F^{-1}u\F$. So it is the quantum subgroup of $\hat S_N^+:=\F^{-1}S_N^+\F$ with respect to the relation $\hat u\hat A=\hat A\hat u$.

\subsection{2-cocycle twists}
\label{secc.twist}

Let $G$ be a quantum group. A unitary 2-cocycle on $\Olg(G)$ is a convolution invertible linear map $\Olg(G)\otimes \Olg(G)\to\C$ satisfying certain axioms. Such a 2-cocycle defines a quantum group -- the \emph{cocycle twist} $\breve G$ \cite{Doi93}. It holds that this new quantum group is monoidally equivalent with the original one \cite{Sch96}. We will use a special class of 2-cocycle here, which will be induced by unitary bicharacters on some finite subgroup of $G$.

\begin{defn}
Let $\Gamma$ be a finite abelian group. A \emph{unitary bicharacter} on $\Gamma$ is a map $\sigma\colon\Gamma\times\Gamma\to\T$ satisfying
$$\sigma(\alpha+\beta,\gamma)=\sigma(\alpha,\gamma)\sigma(\beta,\gamma),\qquad \sigma(\alpha,\beta+\gamma)=\sigma(\alpha,\beta)\sigma(\alpha,\gamma),$$
that is, it has to be a group homomorphism in both entries.
\end{defn}

\begin{rem}[{\cite[Proposition~2.6]{Tam00}}]
Every unitary bicharacter on a finite abelian group $\Gamma$ is a 2-cocycle. It is a 2-coboundary if and only if it is symmetric, i.e.\ $\sigma(\alpha,\beta)=\sigma(\beta,\alpha)$. In fact, every element of the second cohomology $H^2(\Gamma)$ has a unique representative by an alternating bicharacter (i.e. satisfying $\sigma(\alpha,\beta)=\bar\sigma(\beta,\alpha)$).
\end{rem}

Consider a compact matrix quantum group $G\subset O^+(F)$. Suppose that $\hat\Gamma\subset G$. By this, we mean that there is a $*$-homomorphism $\Olg(G)\to\C\Gamma$ mapping $u^i_j\mapsto g_i\delta_{ij}$, where $\{g_i\}$ is some generating set of $\Gamma$. Consider a unitary bicharacter $\sigma$ on $\Gamma$ and denote $\sigma_{ij}:=\sigma(g_i,g_j)$. For a multiindex $\mathbf{i}=(i_1,\dots,i_k)$, denote also $\sigma_\mathbf{i}:=\prod_{1\le a<b\le k}\sigma_{i_ai_b}$.

For every $T\in\RCat(k,l)$, we define $[\breve T]^\mathbf{i}_\mathbf{j}:=\sigma_\mathbf{i}\bar\sigma_\mathbf{j}T^\mathbf{i}_\mathbf{j}$. We put $\breve\RCat(k,l):=\{\breve T\mid T\in\RCat(k,l)\}$.

\begin{prop}
\label{P.Tsigma}
The mapping $T\mapsto\breve T$ is a monoidal unitary functor. That is, $\breve\RCat_G$ is a rigid monoidal $\dag$-category.
\end{prop}
\begin{proof}
Checking that it is a functor (i.e.~$\widebreve{ST}=\breve S\breve T$) and that it is unitary (i.e.~$\widebreve{T^\dag}=\breve T^\dag$) is absolutely straightforward.

Checking that the functor is monoidal is a bit tricky. We need to check that
$$[\breve T\otimes\breve S]^\mathbf{ik}_\mathbf{jl}=\sigma_\mathbf{i}\sigma_\mathbf{k}\bar\sigma_\mathbf{j}\bar\sigma_\mathbf{l}T^\mathbf{i}_\mathbf{j}S^\mathbf{k}_\mathbf{l}=\sigma_\mathbf{ik}\bar\sigma_\mathbf{jl}T^\mathbf{i}_\mathbf{j}S^\mathbf{k}_\mathbf{l}=[\widebreve{T\otimes S}]^\mathbf{il}_\mathbf{jl}$$

Let $a,b,c,d$ be the lengths of the multiindices $\mathbf{i}$, $\mathbf{j}$, $\mathbf{k}$, $\mathbf{l}$ respectively and denote $g_\mathbf{i}:=g_{i_1}\cdots g_{i_a}$ and similarly $g_\mathbf{j}$, $g_\mathbf{k}$, $g_\mathbf{l}$. The difference between left and right hand side is that the right hand side contains in addition the factors $\prod_{x=1}^a\prod_{y=1}^c\sigma_{i_xk_y}=\sigma(g_\mathbf{i},g_\mathbf{k})$ and $\prod_{x=1}^b\prod_{y=1}^d\bar\sigma_{j_xl_y}=\bar\sigma(g_\mathbf{j},g_\mathbf{l})$. 

We know that $T$ and $S$ are intertwiners of $G$ and hence also of $\hat\Gamma\subset G$. In particular, $T^\mathbf{i}_\mathbf{j}g_\mathbf{i}=T^\mathbf{i}_\mathbf{j}g_\mathbf{j}$, $S^\mathbf{k}_\mathbf{l}g_\mathbf{k}=S^\mathbf{k}_\mathbf{l}g_\mathbf{l}$. Consequently, we have $\sigma(g_\mathbf{i},g_\mathbf{k})T^\mathbf{i}_\mathbf{j}S^\mathbf{k}_\mathbf{l}=\sigma(g_\mathbf{j},g_\mathbf{l})T^\mathbf{i}_\mathbf{j}S^\mathbf{k}_\mathbf{l}$. Multiplying by $\bar\sigma(g_\mathbf{j},g_\mathbf{l})$, we get 1, which is what we wanted.
\end{proof}

Applying Woronowicz--Tannaka--Krein, $\breve\RCat_G$ defines a new quantum group $\breve G$ with $\RCat_{\breve G}=\breve\RCat_G$, which we will call the \emph{twist} of $G$. Note that if $G\subset O^+(F)$, then $\breve G\subset O^+(\breve F)$, where $\breve F_i^j=\sigma_{ij}F_i^j$. This new quantum group $\breve G$ is by construction monoidally equivalent with $G$ through the defining functor $T\mapsto\breve T$.

\subsection{Twisting Cayley graphs}
\label{secc.Caytwist}

Let $\Gamma$ be an abelian group. It is well known that $\Gamma$ acts on $\Cay(\Gamma,S)$ simply by the group operation (recall that $\Gamma$ is also the vertex set of $\Cay(\Gamma,S)$). So, we can write $\Gamma\subset\Aut\Cay(\Gamma,S)\subset\Aut^+\Cay(\Gamma,S)$. We may ask how exactly does $\Gamma$ embed into the automorphism group. In terms of matrices, an element $\alpha\in\Gamma$ acts through a permutation matrix $\Pi_\alpha$ with entries $[\Pi_\alpha]^\beta_\gamma=\delta_{\beta,\alpha+\gamma}$.

However, in order to construct the 2-cocycle twist, it would be more convenient if $\Gamma$ was included as the diagonal subgroup. This can be achieved using the Fourier transform on $\Gamma$: We define $\hat\Pi_\alpha:=\F^{-1}\Pi_\alpha\F$, where $\F$ is the Fourier transformation matrix with entries $\F_\mu^\alpha=\tau_\mu(\alpha)$, where $\tau_\mu$ is the irreducible character corresponding to $\mu\in\Gamma$. The entries of $\hat\Pi$ can be computed as
\begin{equation}\label{eq.FPi}
[\hat\Pi_\alpha]^\mu_\nu=\frac{1}{2^N}\sum_{\beta,\gamma}\delta_{\beta,\alpha+\gamma}\tau_\nu(\beta)\tau_\mu(-\gamma)=\frac{1}{2^N}\sum_\gamma \tau_\nu(\alpha)\tau_{\nu-\mu}(\gamma)=\tau_\nu(\alpha)\delta_{\mu\nu},
\end{equation}
where $N=|\Gamma|$. So, we indeed have $\hat\Gamma\subset\hat G$, where $G=\Aut^+\Cay(\Gamma,S)$ and $\hat G=\F^{-1}G\F$.

Consider now a unitary bicharacter $\sigma\colon\Gamma\times\Gamma\to\C$. Denote $\sigma_{\mu\nu}:=\sigma(\mu,\nu)$. As follows from Proposition~\ref{P.Tsigma}, this defines a new compact matrix quantum group $\breve G$ monoidally equivalent with $\hat G\simeq G$.

How does this quantum group look like? Recall that by definition, the quantum automorphism group $G=\Aut^+\Cay(\Gamma,S)$ of the Cayley graph corresponds to the representation category generated by the classical multiplication $m$ (with $m^\alpha_{\beta\gamma}=\delta_{\alpha\beta\gamma}$), the classical duality morphism $R$ (with $R^{\alpha\beta}=\delta_{\alpha\beta}$) and the classical adjacency matrix $A$; alternatively, we may use the unit $\eta$ or the counit $\eta^\dag$ instead of $R$ (with $\eta^\mu=1$).

As we just described, we need to perform the Fourier transform first, i.e.\ express everything in the Fourier basis $\tau_\mu$. So, let us compute the Fourier transformed intertwiners
$$\hat m^\kappa_{\mu\nu}=\frac{1}{N}\sum_{\alpha,\beta,\gamma}\tau_{\mu}(\alpha)\tau_{\nu}(\beta)\tau_\kappa(-\gamma)\delta_{\alpha\beta\gamma}=\frac{1}{N}\sum_\alpha \tau_{\mu+\nu-\kappa}(\alpha)=\delta_{\kappa,\mu+\nu}$$
Similarly, we compute that
$$\hat R^{\mu\nu}=\frac{1}{N}\delta_{\mu+\nu,0},\qquad \hat\eta^\mu=\delta_{\mu,0}.$$
Finally, as we showed in Section~\ref{secc.Cayley}, the adjacency matrix becomes diagonal $\hat A^\mu_\nu=\lambda_\mu\delta_{\mu\nu}$.

Now, we apply the cocycle twist:
\begin{align*}
\breve m^\kappa_{\mu\nu}&=\bar\sigma_{\mu\nu}\hat m^\kappa_{\mu\nu}=\bar\sigma_{\mu\nu}\delta_{\kappa,\mu+\nu},\\
\breve R^{\mu\nu}&=\sigma_{\mu\nu}\breve R^{\mu\nu}=\frac{1}{N}\sigma_{\mu\nu}\delta_{\mu,-\nu},\\
\breve \eta^\mu&=\hat\eta^\mu=\delta_{\mu,0},\\
\breve A^\mu_\nu&=\hat A^\mu_\nu=\lambda_\mu\delta_{\mu\nu}.
\end{align*}

The quantum group $\breve G$ is given by
$$\Olg(\breve G)=\staralg\left(\breve u^i_j,i,j=1,\dots,N\middle|\begin{matrix}\text{$\breve u$ is unitary, }\breve u\breve F=\breve F\bar{\breve u},\\\breve m(\breve u\otimes \breve u)=\breve u\breve m,\;A\breve u=\breve uA\end{matrix}\right)$$

From Theorem~\ref{T.qiso} it follows that there exists a quantum graph $X$, which is quantum isomorphic with $\Cay(\Gamma,S)$ and whose quantum automorphism group is $\breve G$. We will denote this graph by $\widebreve\Cay(\Gamma,S)$ and call it the \emph{twist} of $\Cay(\Gamma,S)$ by the bicharacter $\sigma$. Now it remains to determine, how actually this graph looks like.

By construction, the 2-cocycle twist preserves the dimension of the fundamental representation, so it will also act on some $N$-dimensional comodule algebra. In other words, it is possible to index the basis of $C(X)$ by the elements $\mu\in\Gamma$. So, let us denote the basis by $\breve\tau_\mu$. The multiplication, the counit, and the adjacency operator is going to be given by precisely the equations above. Consequently, we may define the following:

\begin{defn}
Let $\Gamma$ be a finite abelian group. Let $\sigma$ be a unitary bicharacter on $\Gamma$. We define the quantum space $\breve\Gamma$ as
\begin{equation}\label{eq.twistalg}
C(\breve\Gamma)=C^*(\breve\tau_\mu,\mu\in\Gamma\mid \breve\tau_\mu\breve\tau_\nu=\bar\sigma_{\mu\nu}\breve\tau_{\mu+\nu}).
\end{equation}
\end{defn}

\begin{rem}
This definition is by no means new. The construction of a 2-cocycle twist of a group algebra $\C\Gamma$ of some group $\Gamma$ is very well known, see e.g.~\cite[Section 1]{BB13}. Since for abelian groups $\Gamma$ we have $\C\Gamma\simeq C(\Gamma)$, we can interpret the twisted group algebra also as a twisted function algebra. The construction is a special case of a 2-cocycle twist construction for arbitrary Hopf comodule algebras, see e.g.~\cite[Section~7.5]{Mon93}.
\end{rem}

\begin{prop}
The mapping $\breve\eta^\dag\colon C(\breve\Gamma)\to\C$ defined by $\breve\eta^\dag(\breve\tau_\mu)=N\delta_{\mu,0}$ turns $C(\breve\Gamma)$ into a symmetric Frobenius algebra and hence it defines a finite quantum space.
\end{prop}
\begin{proof}
We have to show that the bilinear form $\breve R=\breve\eta^\dag\circ \breve m$ is indeed symmetric. Its entries are $\breve R_{\mu\nu}=\bar\sigma_{\mu\nu}R_{\mu\nu}=N\bar\sigma_{\mu\nu}\delta_{\mu+\nu,0}$. So, we only need to check that $\bar\sigma_{\mu,-\mu}=\bar\sigma_{-\mu,\mu}$. This is indeed true: Since $\sigma$ is a unitary bicharacter, both are actually equal to $\sigma_{\mu\mu}$.
\end{proof}

\begin{rem}
$C(\breve\Gamma)$ is not a Hopf algebra, so $\breve\Gamma$ is not a group nor a quantum group. It is only a quantum space.
\end{rem}

\begin{defn}
Let $\Gamma$ be a finite abelian group, $\sigma$ a unitary bicharacter on $\Gamma$ and $S\subset\Gamma$. For $\mu\in\Gamma$, denote by $\lambda_\mu=\sum_{\theta\in S}\tau_\mu(-\theta)$ the corresponding eigenvalue of the classical Cayley graph $\Cay(\Gamma,S)$. We define the $\sigma$-twisted Cayley graph $\widebreve\Cay(\Gamma,S)$ to be the quantum graph with underlying C*-algebra $C(\breve\Gamma)$ and adjacency operator $\breve A$ defined by $\breve A\breve\tau_\mu=\lambda_\mu\breve\tau_\mu$.
\end{defn}

\begin{thm}\label{T3}
Let $\Gamma$ be a finite abelian group, $\sigma$ a unitary bicharacter on $\Gamma$ and $S\subset\Gamma$. The quantum graph $\widebreve\Cay(\Gamma,S)$ is quantum isomorphic to the classical Cayley graph $\Cay(\Gamma,S)$. Its quantum automorphism group is the 2-cocycle twist $\breve G$ of the quantum automorphism group $G$ of the classical Cayley graph $\Cay(\Gamma,S)$.
\end{thm}
\begin{proof}
The theorem follows from the considerations at the beginning of this subsection. We constructed the quantum graph exactly by applying the monoidal equivalence $\Rep G\to \Rep \breve G$ on the multiplication $m$, counit $\eta^\dag$, and adjacency operator $A$ corresponding to the classical Cayley graph $\Cay(\Gamma,S)$. Therefore, by Theorem~\ref{T.qiso}, this must define a quantum graph, which is quantum isomorphic to the original one and its quantum symmetries are given by $\breve G$.
\end{proof}

\begin{rem}
As we mentioned in Proposition~\ref{P.qisomulti}, the quantum graph $\widebreve\Cay(\Gamma,S)$ has the same properties regarding having multiple edges, being directed, or having loops as the classical graph $\Cay(\Gamma,S)$. This means the following:
\begin{enumerate}
\item $\widebreve\Cay(\Gamma,S)$ has no multiple edges (unless $S$ is a multiset containing some generators more than once).
\item $\widebreve\Cay(\Gamma,S)$ is undirected if and only if $S$ is closed under the group inversion.
\item If $0\not\in S$, then $\widebreve\Cay(\Gamma,S)$ has no loops. Otherwise it has a loop at every vertex.
\end{enumerate}
\end{rem}

\begin{rem}
Since we do not change the adjacency operator, the resulting graph $\widebreve\Cay(\Gamma,S)$ is isomorphic to the original Cayley graph $\Cay(\Gamma,S)$ if and only if $C(\Gamma)\simeq C(\breve\Gamma)$. Since quantum isomorphisms preserve the dimension of the algebra, this is equivalent to $C(\breve\Gamma)$ being commutative. The commutativity relations on $C(\breve\Gamma)$ are of the form $\breve\tau_\mu\breve\tau_\nu=\bar\sigma_{\mu\nu}\sigma_{\nu\mu}\breve\tau_\nu\breve\tau_\mu$. So, the algebra is commutative if and only if $\sigma$ is symmetric i.e.\ a coboundary.
\end{rem}

\subsection{Comparison with Musto--Reutter--Verdon construction}

As we mentioned in the introduction to Section~\ref{sec.twist}, our construction is actually a special case of the twisting procedure introduced in \cite{MRV19}. So, let us sketch the relationship.

Given a graph $\Gr$, the construction of \cite{MRV19} starts with a subgroup $L\subset\Aut\Gr$ equipped with a 2-cocycle $\sigma$. Then the twist $\C L^\sigma$ of the group algebra $\C L$ is constructed (the same way as we did when we defined $C(\breve\Gamma)$). But now the point is that the group $L$ is in general somewhat unrelated to the underlying set of points of the graph $\Gr$. Based on these data, the article \cite{MRV19} provides a clever but somewhat complicated way of how to construct a quantum space $X$ and an adjacency matrix $A\colon l^2(X)\to l^2(X)$, which define a quantum graph quantum isomorphic to the original one.

Now we claim that for $\Gr=\Cay(\Gamma,S)$ and $L=\Gamma$, the resulting quantum space will be again the twisted group algebra we started with (up to isomorphism), so we obtain exactly our construction. Let us go through the steps of their construction and apply them to our case $L=\Gamma$ to show this.

So, take a unitary bicharacter $\sigma$ on $\Gamma$. Without loss of generality, we can assume that $\sigma$ is alternating. For simplicity, we will assume that the twisted group algebra $\C\Gamma^\sigma$ is simple. (In fact, we have to assume this since the construction of \cite{MRV19} relies on this fact. Otherwise, we would need to take as suitable subgroup $L\subset\Gamma$.) Being simple means that $\C\Gamma^\sigma$ is a matrix algebra. Denote by $H$ the Hilbert space on which it acts, so $\C\Gamma^\sigma=\End H$, and by $U_\mu$, $\mu\in\Gamma$ the corresponding unitary generators.\footnote{Such a concrete realization is also known as a \emph{nice unitary error basis}, see \cite{KR03}.} 

The construction of \cite{MRV19} starts by defining a certain map $H\otimes H^*\otimes l^2(\Gamma)\to l^2(\Gamma)\otimes H\otimes H^*$ by \cite[Eq.~(68)]{MRV19}. This map can be rotated to a linear operator $\pi\colon H^*\otimes l^2(\Gamma)\otimes H\to H^*\otimes l^2(\Gamma)\otimes H$. In our case, it is of the form
$$\pi={1\over N}\sum_{\alpha\in\Gamma}\bar U_\alpha\otimes\phi_\alpha\otimes U_\alpha,$$
where $\phi_\alpha\in\Aut\Cay(\Gamma,S)$ denotes the translation automorphism of the Cayley graph induced by the group element $\alpha$, i.e.\ $\phi_\alpha(\beta)=\beta+\alpha$. Now it is convenient to use the identification $H^*\otimes H\simeq\End H$, so we can express $\pi$ as an operator on $\End H\otimes l^2(\Gamma)$. We can then conveniently express this operator in the basis $(U_\alpha)$ as
$$U_\mu\otimes\delta_\beta\mapsto\frac{1}{N}\sum_{\alpha\in\Gamma}U_\alpha^{\dag}U_\mu U_\alpha\otimes\delta_{\alpha+\beta}.$$
In fact, $\pi$ is an orthogonal projection, which can be easily checked by direct computation. As we already mentioned several times, it is more convenient to express everything in the Fourier basis $(\tau_\mu)$ rather than the basis of the canonical projections $(\delta_\beta)$. The action of $\Gamma$ was already expressed in Equation~\eqref{eq.FPi} as $\phi_\alpha(\tau_\nu)=\tau_\nu(\alpha)\tau_\nu$. So, in the basis $(U_\mu\otimes\tau_\nu)$ the projection $\pi$ acts diagonally by
$$U_\mu\otimes\tau_\nu\mapsto
\rho_{\mu\nu}\,U_\mu\otimes\tau_\nu,\quad
\rho_{\mu\nu}=\frac{1}{N}\sum_{\alpha\in\Gamma}\bar\sigma_{\mu\alpha}\sigma_{\alpha\mu}\tau_\nu(\alpha).$$
Now, one can compute that $\sum_\nu\rho_{\mu\nu}=1=\sum_\mu\rho_{\mu\nu}$ (in the latter, we use the simplicity of $\C\Gamma^\sigma$), so $(\rho_{\mu\nu})$ is a permutation matrix. Denoting the corresponding permutation by $\rho\colon\Gamma\to\Gamma$, we can write $\rho_{\mu\nu}=\delta_{\mu\rho(\nu)}$. As a consequence, the range of $\pi$ has dimension $N$. Actually, one can prove that $\rho$ is an automorphism of the group $\Gamma$.

Second step is defining $l^2(X):=\mathop{\rm Ran}\pi$ to be the range of $\pi$ and thus decomposing $\pi$ as $\pi=\iota\iota^\dag$ with $\iota\colon l^2(X)\to H^*\otimes l^2(\Gamma)\otimes H\simeq\End H\otimes l^2(\Gamma)$ being the corresponding isometry. This is quite straightforward here: We just denote the images of the basis vectors by $\breve\tau_{\nu}:=U_{\rho(\nu)}\otimes\tau_\nu\in l^2(X)$. Now, we rotate the $\iota$ to obtain a map $P\colon H\otimes l^2(X)\to l^2(\Gamma)\otimes H$, which is now given by $x\otimes\breve\tau_\nu\mapsto\tau_\nu\otimes U_{\rho(\nu)}x$ for any $x\in H$ and $\nu\in\Gamma$.

Finally, the article \cite{MRV19} claims that $P$ is unitary and hence defines a ``quantum bijection'' which can be used to define a quantum graph, which is quantum isomorphic to the original one in the following way: We define $P^{\OT k}\colon H\otimes (l^2(X))^{\otimes k}\to (l^2(X))^{\otimes k}\otimes H$ by $k$-times composing $P$ in the $H$ tensor factor. For any $T\in\RCat_G(k,l)$, $G=\Aut^+\Cay(\Gamma,S)$, we define $\breve T\colon l^2(X)^{\otimes k}\to l^2(X)^{\otimes l}$ by $\breve T:=\frac{1}{N}\Tr_H(P^{\OT l\,\dag}TP^{\OT k})$, which induces a new monoidal $\dag$-category isomorphic to $\RCat_G$. In particular, it defines a multiplication $\breve m$, a unit $\breve\eta$, and an adjacency matrix $\breve A$ on $l^2(X)$.

In our case, we get $\breve m(\breve\tau_\mu\otimes\breve\tau_\nu)=\frac{1}{N}\Tr(U_{\rho(\mu+\nu)}^\dag U_{\rho(\nu)} U_{\rho(\mu)})\breve\tau_{\mu+\nu}=\bar\sigma_{\rho(\nu)\rho(\mu)}\breve\tau_{\mu+\nu}$, which is up to the order and up to the automorphism $\rho$ exactly the multiplication in the twisted group algebra we started with. A straightforward computation also shows that $\breve\eta=\breve\tau_0$ and $\breve A\breve\tau_\mu=\lambda_\mu\breve\tau_\mu$, so we indeed exactly recover our construction.

\section{Twisting $\Z_n\times\Z_n$}
\label{sec.rook}

The aim of this section is to construct the quantum space $M_n$ as a twist of the classical space $X_{n^2}$. For this purpose, we choose the group $\Gamma=\Z_n\times\Z_n$ as it has exactly the correct size $N=n^2$. This allows not only to construct quantum analogues of the Cayley graphs over $\Gamma$ (such as the rook's graph), but it also incidentally allows to show that
\begin{itemize}
\item $M_n$ is quantum isomorphic to $X_{n^2}$,
\item consequently, any two quantum spaces $X$ and $Y$ are quantum isomorphic if and only if the number of vertices coincide $\dim C(X)=\dim C(Y)$,
\item the quantum symmetry group of $M_n$ is $PO_n^+$,
\item and to derive explicitly the monoidal equivalence $PO_n^+\sim S_{n^2}^+$.
\end{itemize}

None of these results is actually new \cite{Ban99,Ban02,RV10,GProj}, but the twisting procedure provides a new viewpoint on these results.

\subsection{The comodule algebra twist}
Denote by $\epsilon_1,\epsilon_2$ the canonical generators of $\Gamma=\Z_n\times\Z_n$ and by $\tau_1,\tau_2$ the corresponding generators of $C(\Gamma)\simeq\C\Gamma$. That is,
$$C(\Gamma)=C^*(\tau_1,\tau_2\mid \tau_1^n=\tau_2^n=\tau_1\tau_1^*=\tau_1^*\tau_1=\tau_2\tau_2^*=\tau_2^*\tau_2=1,\;\tau_1\tau_2=\tau_2\tau_1)$$

We can define a unitary bicharacter $\sigma$ on $\Gamma$ by $\sigma(\epsilon_1,\epsilon_2)=\sigma(\epsilon_1,\epsilon_1)=\sigma(\epsilon_2,\epsilon_2)=1$ and $\sigma(\epsilon_2,\epsilon_1)=\omega$, where $\omega={\rm e}^{2\pi\im/n}$ is a primitive $n$-th root of unity. One can easily compute that this extends to the whole group $\Gamma$ as
$$\sigma_{abcd}:=\sigma(a\epsilon_1+b\epsilon_2,c\epsilon_1+d\epsilon_2)=\omega^{bc}.$$

Applying the corresponding cocycle twist, we get the comodule algebra
$$C(\breve\Gamma)=C^*(\breve\tau_1,\breve\tau_2\mid \breve\tau_1^n=\breve\tau_2^n=\breve\tau_1\breve\tau_1^*=\breve\tau_1^*\breve\tau_1=\breve\tau_2\breve\tau_2^*=\breve\tau_2^*\breve\tau_2=1,\;\breve\tau_1\breve\tau_2=\omega\,\breve\tau_2\breve\tau_1).$$

\begin{lem}
The algebra $C(\breve\Gamma)$ is isomorphic to $M_n(\C)$ via
\begin{align*}
\phi(\tau_1^a\tau_2^b)&=
\begin{pmatrix}
&&&1\\
&&&&\ddots\\
&&&&&\omega^{ab}\\
\omega^{a(b+1)}\\
&\ddots\\
&&\omega^{an}\\
\end{pmatrix}\\
\phi^{-1}(e_{ij})&=
\frac{1}{n}\sum_{a}\omega^{-ia}\tau_1^a\tau_2^{i-j}
\end{align*}
\end{lem}
\begin{proof}
It is straightforward to check that the matrices
$$\phi(\breve\tau_1)=
\begin{pmatrix}
1\\
&\omega\\
&&\omega^2\\
&&&\ddots\\
&&&&\omega^{n-1}
\end{pmatrix}
,\qquad
\phi(\breve\tau_2)=
\begin{pmatrix}
  &   &&& 1\\
1 &   &&&\\
  & 1 &&&\\
&&\ddots&&\\
&&&1& 
\end{pmatrix}
$$
satisfy the defining relations of $C(\breve\Gamma)$.

Secondly, we need to check that the formulas for $\phi$ and $\phi^{-1}$ are indeed inverse to each other. In order to do so, let us denote the entries of $\phi$ and $\phi^{-1}$, so $\phi^{ij}_{ab}=\delta_{b,i-j}\omega^{ia}$, $[\phi^{-1}]^{ab}_{ij}=\frac{1}{n}\delta_{b,i-j}\omega^{-ia}$. Now, we can compute
\[[\phi\phi^{-1}]^{ij}_{kl}=\frac{1}{n}\sum_{a,b}\delta_{b,i-j}\delta_{b,k-l}\omega^{ia}\omega^{-ka}=\frac{1}{n}\delta_{i-j,k-l}\sum_a\omega^{a(i-k)}=\delta_{ik}\delta_{jl}.\qedhere\]
\end{proof}

As a consequence, we can prove that all quantum spaces with $N$ points are quantum isomorphic. This result is actually implicitly contained already in the work of Banica \cite{Ban99,Ban02} and later it was formulated as an explicit statement in \cite[Theorem~4.7]{RV10}. See also \cite{GProj}.

\begin{prop}
The quantum spaces $X_{n^2}$ and $M_n$ are quantum isomorphic.
\end{prop}
\begin{proof}
We just proved that $M_n$ is a twist of $X_{n^2}$. By Theorem~\ref{T3}, they must be quantum isomorphic.
\end{proof}

\begin{thm}
All quantum spaces $X$ with fixed number of vertices $N=\dim C(X)$ are mutually quantum isomorphic.
\end{thm}
\begin{proof}
Given a quantum space $X$ with $\dim C(X)=N$, we need to prove that $X$ is quantum isomorphic to $X_N$. We know that $X$ must be of the form $C(X)=\bigoplus_i M_{n_i}(\C)$. Denote by $m_i,\eta_i$ the multiplication and counit on $C(X_{n_i^2})$ and by $\tilde m_i,\tilde\eta_i$ the multiplication and counit on $M_{n_i}(\C)$. We know that each summand $M_{n_i}$ is quantum isomorphic with $X_{n_i^2}$. That is, there is a monoidal $*$-isomorphism of the representation categories $\Rep S_{M_n}^+$ and $\Rep S_{n_i^2}^+$ mapping $\tilde m_i\leftrightarrow m_i$ and $\tilde\eta_i\leftrightarrow\eta_i$. Moreover, the representation categories are by definition generated by $\tilde m_i,\tilde\eta_i$, resp. $m_i,\eta_i$. Consequently, if we construct the representation categories generated by $\tilde m:=\bigoplus \tilde m_i$ and $\tilde \eta:=\bigoplus \tilde\eta_i$, resp. $m:=\bigoplus m_i$ and $\eta:=\bigoplus\eta_i$, then those must be monoidally $*$-isomorphic as well. But those are exactly the representation categories associated to $C(X)=\bigoplus M_{n_i}(\C)$ and $C(X_N)=\bigoplus C(X_{n_i^2})$. Hence $X$ is quantum isomorphic with $X_N$.
\end{proof}

\subsection{Quantum automorphism group of $M_n$}

It was proven already in \cite[Corollary~4.1]{Ban99} that the quantum automorphism group of $M_n$ is the projective free orthogonal quantum group $PO_n^+$. We derive this result again here. As follows from the considerations in the previous subsection, the quantum automorphism group has to be monoidally equivalent with the free symmetric quantum group $S_{n^2}^+$. We derive explicitly this monoidal equivalence.

So let us follow the considerations of Section~\ref{secc.Caytwist}. In the following text, instead of indexing the intertwiners using Greek letters $\alpha\in\Gamma=\Z_n\times\Z_n$, we will use pairs of indices $a,b$, which will stand for the element $\tau_1^a\tau_2^b$. So, we can write
$$\hat m^{xy}_{abcd}=\delta_{a+c,x}\delta_{b+d,y},\qquad\hat R^{abcd}=\frac{1}{n^2}\delta_{a,-c}\delta_{b,-d}.$$
All indices are taken modulo $n$ here.

Now, we apply the cocycle twist:
\begin{align*}
\breve m^{xy}_{abcd}&=\bar\sigma_{abcd}\delta_{a+c,x}\delta_{b+d,y}=\omega^{-bc}\delta_{a+c,x}\delta_{b+d,y},\\
\breve R^{abcd}&=\frac{1}{n^2}\sigma_{abcd}\delta_{a,-c}\delta_{b,-d}=\frac{1}{n^2}\omega^{bc}\delta_{a,-c}\delta_{b,-d}.
\end{align*}

We could write down the definition of $\breve G$ now. However, we would not yet recognize the quantum group $PO_n^+$, because we expressed everything in the twisted Fourier basis $(\breve\tau_1^a\breve\tau_2^b)$. Instead, we would like to use the canonical orthonormal basis for $M_n$, which is $(\frac{1}{\sqrt{n}}\,e_{ij})$. We will not use the basis $(e_{ij})$ since it is not orthonormal (recall Example~\ref{ex.Mn}) and hence the involution $\dag$ would not act simply as conjugate transposition, which might be confusing. So, denote $\tilde m:=n^{-1/2}\,\phi\circ \breve m\circ(\phi^{-1}\otimes\phi^{-1})$ and $\tilde R:=n(\phi\otimes\phi)\breve R$ and compute them:
\begin{align*}
\tilde R^{ijkl}&=n\sum_{a,b,c,d}\phi^{ij}_{ab}\phi^{kl}_{cd}\breve R^{abcd}=\frac{1}{n}\sum_{a,b,c,d}\delta_{b,i-j}\omega^{ia}\delta_{d,k-l}\omega^{kc}\omega^{bc}\delta_{a,-c}\delta_{b,-d}=\\
&=\frac{1}{n}\delta_{i-j,l-k}\sum_a \omega^{a(j-k)}=\delta_{jk}\delta_{il}\displaybreak[1]\\
\tilde m^{rs}_{ijkl}&=\frac{1}{\sqrt{n}}\sum_{a,b,c,d,x,y}\phi^{rs}_{xy}[\phi^{-1}]^{ab}_{ij}[\phi^{-1}]^{cd}_{kl}\breve m^{xy}_{abcd}\\
&=\frac{1}{n^{5/2}}\sum_{a,b,c,d,x,y}\delta_{y,r-s}\omega^{rx}\delta_{b,i-j}\omega^{-ia}\delta_{d,k-l}\omega^{-kc}\omega^{-bc}\delta_{a+c,x}\delta_{b+d,y}\\
&=\frac{1}{n^{5/2}}\delta_{r-s,i-j+k-l}\sum_{a,c}\omega^{a(r-i)}\omega^{c(r-k-i+j)}=\frac{1}{\sqrt{n}}\delta_{ri}\delta_{kj}\delta_{sl}
\end{align*}

Readers familiar with the idea of modelling the representation category of $S_N^+$ using partitions may appreciate rewriting the monoidal $*$-isomorphism $R\mapsto\tilde R$, $m\mapsto\tilde m$ using partitions as
$$
{\def\Pwidth{0.2}\Laa}\mapsto\LPartition{}{0.65:0.9,2.1;0.5:1.1,1.9},
\qquad
{\def\Pwidth{0.2}\mergepart}\mapsto\frac{1}{\sqrt n}
\Partition{
\Pblock 1to0.60:1.1,1.9
\Pline (0.9,1)(0.9,0.45)
\Pline (0.9,0.45)(1.4,0.45)
\Pline (1.4,0.45)(1.4,0)
\Pline (2.1,1)(2.1,0.45)
\Pline (2.1,0.45)(1.6,0.45)
\Pline (1.6,0.45)(1.6,0)
},
$$
which corresponds to a well known procedure of ``fattening partitions''. Here, the partitions are (exceptionally) meant to be read from top to bottom, but reading them from bottom to top actually also produces a true statement. The partition on the left-hand side is intentionally drawn in thick lines in order to emphasise the fact that each point/object is actually represented by a pair of indices in the equations. See also \cite[Section~4]{KS08}, where this isomorphism was proven in a diagrammatic way.

\subsection{Quantum rook's graph}

Equipping $\Gamma=\Z_n\times\Z_n$ with any set $S\subset\Gamma$, we get a Cayley graph $\Cay(\Gamma,S)$. Using the twisting procedure above, we obtain a quantum analogue of this Cayley graph, which is defined over the quantum space $\breve\Gamma=M_n$. For instance, we can start with the \emph{rook's graph}, which is the Cayley graph of $\Z_n\times\Z_n$ with respect to the generating set $\{a\epsilon_1\}_{a=1}^{n-1}\cup\{b\epsilon_2\}_{b=1}^{n-1}$.

Using formula~\eqref{eq.lambda}, we can compute the spectrum of rook's graph, which also gives us the adjacency matrix in the Fourier basis $\hat A^{ab}_{cd}=\delta_{ac}\delta_{bd}\lambda_{ab}$, where
$$\lambda_{ab}=\sum_{i=1}^{n-1}\omega^{ia}+\sum_{j=1}^{n-1}\omega^{jb}=n\delta_{a,0}+n\delta_{b,0}-2.$$

By twisting, the adjacency matrix does not change, so $\breve A=\hat A$. Finally, we can compute its entries in the basis $(e_{ij})$ as
\begin{align*}
\tilde A^{ij}_{kl}
&=\sum_{a,b}\phi^{ij}_{ab}[\phi^{-1}]^{ab}_{kl}\lambda_{ab}
 =\frac{1}{n}\sum_{a,b}\delta_{b,i-j}\delta_{b,k-l}\omega^{a(i-k)}(n\delta_{a,0}+n\delta_{b,0}-2)\\
&=\delta_{i-j,k-l}+n\delta_{ijkl}-2\delta_{ik}\delta_{jl}.
\end{align*}

The quantum space $M_n$ equipped with the adjacency matrix $\tilde A$ given by the equation above may be called the \emph{quantum rook's graph}. It is quantum isomorphic to the ordinary rook's graph.

Note that the quantum automorphism group of the ordinary rook's graph was recently obtained in \cite[Section~6]{GroAbSym}. It is the wreath product $S_n^+\wr\Z_2$. Performing the 2-cocycle twist and adapting the proof from \cite{GroAbSym}, one could derive the structure of the twisted quantum automorphism group.

\section{Anticommutative cube-like graphs}
\label{sec.cube}

\subsection{General idea}

Consider the group $\Gamma=\Z_2^n$ and denote by $\epsilon_1,\dots,\epsilon_n$ the canonical generators of this group. Recall that the algebra of functions $C(\Z_2^n)\simeq\C\Z_2^n$ is spanned by the characters $\tau_\alpha$ and, in particular, it is generated by the characters $\tau_1,\dots,\tau_n$. Here $\tau_i=\tau_{\epsilon_i}$, so $\tau_i(\alpha)=(-1)^{\alpha_i}$ for any $\alpha\in\Gamma$. That is,
$$C(\Z_2^n)=C^*(\tau_1,\dots,\tau_n\mid \tau_i=\tau_i^*,\;\tau_i^2=1,\;\tau_i\tau_j=\tau_j\tau_i).$$

Choosing an appropriate bicharacter, we can make the generators $\tau_i$ anticommute. We define a bicharacter $\sigma$ on $\Z_2^n$ as follows:
\begin{equation}\label{eq.sigma1}
\sigma_{ij}:=\sigma(\epsilon_i,\epsilon_j)=\begin{cases}-1&\text{if $i>j$,}\\+1&\text{if $i\le j$,}\end{cases}
\end{equation}

It is easy to see and it was already mentioned in \cite{AM02} that twisting the algebra $C(\Z_2^n)$ as a comodule algebra using formula \eqref{eq.twistalg}, we obtain exactly the \emph{Clifford algebra}.

\begin{defn}
{\em Clifford algebra} with $n$ generators is defined by generators and relations as follows
$$\Cl_n=C^*(\breve\tau_1,\dots,\breve\tau_n\mid \breve\tau_i=\breve\tau_i^*,\;\breve\tau_i^2=1,\;\breve\tau_i\breve\tau_j=-\breve\tau_j\breve\tau_i).$$
\end{defn}

We are going to use a similar notation for the basis of the Clifford algebra as we did for $C(\Z_2^n)$. Namely, given $\alpha=\epsilon_{i_1}+\cdots+\epsilon_{i_k}\in\Z_2^n$ with $1\le i_1<\cdots<i_k\le n$, we denote $\breve\tau_\alpha:=\breve\tau_{i_1}\cdots\breve\tau_{i_k}$. Then the elements $(\breve\tau_\alpha)$ form a linear basis of $\Cl_n$.

\begin{prop}
Consider $S\subset\Z_2^n$. Then the adjacency operator $A\colon\Cl_n\to\Cl_n$ defined by $A\breve\tau_\mu=\lambda_\mu\tau_\mu$ with $\lambda_\mu=\sum_{\theta\in S}(-1)^{\mu_1\theta_1+\cdots+\mu_n\theta_n}$ defines a quantum graph $\widebreve{\Cay}(\Z_2^n,S)$, which is quantum isomorphic to the classical Cayley graph $\Cay(\Z_2^n,S)$.
\end{prop}
\begin{proof}
Follows directly from Theorem~\ref{T3}. Notice that the formula for $\lambda_\mu$ exactly corresponds to the formula~\eqref{eq.lambda} for eigenvalues of the classical Cayley graph $\Cay(\Gamma,S)$.
\end{proof}

Classical Cayley graphs $\Cay(\Gamma,S)$ with $\Gamma=\Z_2^n$ are called \emph{cube-like}, because they have $2^n$ vertices, which can be naturally arranged to form vertices of $n$-dimensional hypercube.

\begin{defn}
\label{D.anticom}
The quantum graph $\widebreve{\Cay}(\Z_2^n,S)$ constructed in the previous proposition will be called the \emph{anticommutative deformation} of the original cube-like graph $\Cay(\Gamma,S)$.
\end{defn}

\begin{rem}
There is quite an illustrative way how to view these graphs. The classical hypercube graph and any other cube-like graph can be ``realized'' by putting an actual hypercube into the space $\R^n$. The vertices of the hypercube are placed to positions with coordinates $(x_1,\dots,x_n)$, $x_i=\pm 1$. The characters $\tau_i$ can then be understood as the coordinate functions in the $n$ dimensions assigning those values $\pm 1$ to the vertices. By twisting the coordinate algebra $C(\Z_2^n)$ into the Clifford algebra $\Cl_n$, we essentially replace the Euclidean space $\R^n$ by a certain quantum space, where the coordinate functions mutually anticommute.
\end{rem}

\subsection{Anticommutative deformations}
We say that a~matrix $u$ has \emph{anticommutative entries} if the following relations hold
$$u^i_ku^j_k=-u^j_ku^i_k,\quad u^k_iu^k_j=-u^k_ju^k_i,\quad u^i_ku^j_l=u^j_lu^i_k$$
assuming $i\neq j$ and $k\neq l$.

As an example, let us mention the \emph{anticommutative orthogonal quantum group}
$$\Olg(O_N^{-1})=\staralg(u^i_j\mid\text{$u=\bar u$, $u$ orthogonal, $u$ anticommutative}).$$

There is a whole theory about $q$-deformations of classical groups, where taking $q=1$ gives the classical case and $q=-1$ gives usually the anticommutative one, see \cite{KS97} for more details.

\subsection{Anticommutative hypercube graph}

Let us now look on three examples of cube-like graphs and their deformations more in detail. First, the hypercube graph.

\begin{defn}
The anticommutative deformation of the hypercube graph in the sense of Def.~\ref{D.anticom} will be called the \emph{anticommutative hypercube graph} and denoted by $\breve Q_n:=\widebreve\Cay(\Z_2^n,\{\epsilon_i\}_{i=1}^n)$ 
\end{defn}

The quantum automorphism group of the classical hypercube graph $Q_n$ is $O_n^{-1}$ \cite{BBC07}. It is also known that $O_n^{-1}$ is a twist of $O_n$ with respect to exactly the same unitary bicharacter \eqref{eq.sigma1} \cite{BBC07,GWgen}. Applying the twist again, we are able to ``undo'' the anticommutative deformation and obtain $O_n$ again. Consequently, we have the following:

\begin{prop}
The anticommutative hypercube graph $\breve Q_n$ is quantum isomorphic to the classical hypercube graph $Q_n$. Its quantum automorphism group is $O_n$.
\end{prop}

\subsection{Anticommutative folded hypercube}
First, we recall the situation for the classical case, see \cite[Section~5]{GroAbSym} for more details.

The folded hypercube $FQ_{n+1}$ is a quotient graph of the hypercube $Q_{n+1}$. This means that we can view the function algebra $C(FQ_{n+1})=C(\Z_2^n)$ as a subalgebra of $C(Q_{n+1})=C(\Z_2^{n+1})$ spanned by $\tau_\mu$ of even degree. Namely we use the embedding $\tau_\mu\mapsto\tau_\mu$ when $\deg\mu$ is even and $\tau_\mu\mapsto\tau_\mu\tau_{n+1}$ if $\deg\mu$ is odd, where $\deg\mu$ is the number of \emph{ones} in the tuple $\mu=(\mu_1,\dots,\mu_n)$.

Recall that $O_{n+1}^{-1}$ is the quantum symmetry group of $Q_{n+1}$. The action is via $\tau_i\mapsto\sum_j\tau_j\otimes u^j_i$, where $u$ is the fundamental representation of $O_{n+1}^{-1}$. We define the \emph{projective version} $PO_{n+1}^{-1}$ of $O_{n+1}^{-1}$ to be the quantum group with associated Hopf algebra $\Olg(PO_{n+1}^{-1})$ defined as the subalgebra of $\Olg(O_{n+1}^{-1})$ generated by elements $u^i_ju^k_l$, i.e.\ consisting of even polynomials in $u^i_j$. In \cite[Theorem~5.1]{GroAbSym}, we proved that the quantum symmetry group of $FQ_{n+1}$ is $PO_{n+1}^{-1}$. The corresponding action is given exactly by restricting the above mentioned coaction $C(Q_{n+1})\to C(Q_{n+1})\otimes\Olg(O_{n+1}^{-1})$ to the corresponding subalgebra, so we get $C(FQ_{n+1})\to C(FQ_{n+1})\otimes\Olg(PO_{n+1}^{-1})$.

One can show that the folded hypercube can actually be equivalently constructed by adding the edge connecting each two opposite vertices to the $n$-dimensional hypercube $Q_n$. This means that $FQ_{n+1}$ is a Cayley graph of $\Z_2^n$ with respect to the generating set $\{\epsilon_1,\dots,\epsilon_n,\iota\}$, where $\iota=\epsilon_1+\cdots+\epsilon_n=(1,\dots,1)$. Its spectrum is computed using formula~\eqref{eq.lambda} as
\begin{equation}\label{eq.FQlambda}
\lambda_\mu=\sum_{i=1}^n\tau_\mu(\epsilon_i)+\tau_\mu(\epsilon_1+\cdots+\epsilon_n)=n-2\deg\mu+(-1)^{\deg\mu}=n+1-4\lceil 1/2\,\deg\mu\rceil.
\end{equation}

Now we would like to show that all this holds also for the anticommutative deformation.

\begin{defn}
The anticommutative deformation of the folded hypercube graph $FQ_{n+1}$ in the sense of Def.~\ref{D.anticom} will be called the \emph{anticommutative folded hypercube graph} and denoted by $\widebreve{FQ}_{n+1}:=\widebreve\Cay(\Z_2^n,\{\epsilon_i\}_{i=1}^n\cup\{\epsilon_1+\cdots+\epsilon_n\})$ 
\end{defn}

First, we have by construction the following result.

\begin{prop}
The anticommutative folded hypercube $\widebreve{FQ}_{n+1}$ is quantum isomorphic to the classical folded hypercube $FQ_{n+1}$. Assuming $n\neq 1,3$, it has quantum automorphism group $PO_{n+1}$.
\end{prop}
\begin{proof}
The first sentence follows by construction directly from Theorem \ref{T3}. As we already mentioned, $O_{n+1}$ is a 2-cocycle twist of $O_{n+1}^{-1}$ with respect to the bicharacter \eqref{eq.sigma1}. Consequently, the same must hold also for $PO_{n+1}$ and $PO_{n+1}^{-1}$ since those are given just by restricting to a certain subalgebra.
\end{proof}

Secondly, we would like to prove that the anticommutative folded hypercube $\widebreve{FQ}_{n+1}$ can be obtained from the anticommutative hypercube $\breve Q_{n+1}$ by some folding, i.e.\ as a quotient graph. We can then virtually say that \emph{anticommutative folded hypercube equals folded anticommutative hypercube}, i.e. $\widebreve{FQ}_{n+1}=F\breve Q_{n+1}$.

\begin{prop}
The anticommutative folded hypercube $\widebreve{FQ}_{n+1}$ is a quantum quotient graph of $\breve Q_{n+1}$.
\end{prop}
\begin{proof}
We need to establish a suitable embedding of the underlying algebras $\iota\colon\Cl_n\to\Cl_{n+1}$. We claim that the following works: Given $\mu\in\Z_2^n$ of even degree, we map $\breve\tau_\mu\mapsto\breve\tau_\mu$; if $\mu$ has odd degree, we map $\breve\tau_\mu\mapsto\im\breve\tau_\mu\breve\tau_{n+1}$, where $\im$ is the imaginary unit.

First, we need to show that this is indeed an injective $*$-homomorphism. The injectivity is obvious. We check that it behaves well with respect to multiplication:
\begin{align*}
\text{$\mu$, $\nu$ even:}&&\iota(\breve\tau_\mu)\iota(\breve\tau_\nu)&=\breve\tau_\mu\breve\tau_\nu=\iota(\breve\tau_\mu\breve\tau_\nu)\cr
\text{$\mu$ even, $\nu$ odd:}&&\iota(\breve\tau_\mu)\iota(\breve\tau_\nu)&=\breve\tau_\mu\cdot\im\breve\tau_\nu\breve\tau_{n+1}=\im\breve\tau_\mu\breve\tau_\nu\breve\tau_{n+1}=\iota(\breve\tau_\mu\breve\tau_\nu)\cr
\text{$\mu$ odd, $\nu$ even:}&&\iota(\breve\tau_\mu)\iota(\breve\tau_\nu)&=\im\breve\tau_\mu\breve\tau_{n+1}\cdot\breve\tau_\nu=\im\breve\tau_\mu\breve\tau_\nu\breve\tau_{n+1}=\iota(\breve\tau_\mu\breve\tau_\nu)\cr
\text{$\mu$, $\nu$ odd:}&&\iota(\breve\tau_\mu)\iota(\breve\tau_\nu)&=\im\breve\tau_\mu\breve\tau_{n+1}\cdot\im\breve\tau_\nu\breve\tau_{n+1}=-\im^2\breve\tau_\mu\breve\tau_\nu\breve\tau_{n+1}^2=\breve\tau_\mu\breve\tau_\nu=\iota(\breve\tau_\mu\breve\tau_\nu)
\end{align*}
And we need to check that it behaves well with respect to $*$. For $\mu$ even: $\iota(\breve\tau_\mu)^*=(\breve\tau_\mu)^*=\iota(\breve\tau_\mu^*)$; for $\mu$ odd: $\iota(\breve\tau_\mu)^*=(\im\breve\tau_\mu\breve\tau_{n+1})^*=-\im\breve\tau_{n+1}\breve\tau_\mu^*=\im\breve\tau_\mu^*\breve\tau_{n+1}=\iota(\breve\tau_\mu^*)$.

Finally, we need to check that $A_{\widebreve{FQ}_{n+1}}=\iota^\dag\circ A_{\breve Q_{n+1}}\circ\iota$. Recall that both adjacency matrices are given by $A\breve\tau_\mu=\lambda_\mu\breve\tau_\mu$, where $\lambda_\mu$ are eigenvalues of the classical Cayley graph, which we already computed in equation~\eqref{eq.FQlambda}. That is, we can express the adjacency operators as matrices in the bases $(\breve\tau_\mu)$ as
$$[A_{\breve{Q}_{n+1}}]^\mu_\nu=\delta_{\mu\nu}(n+1-2\deg\mu),\qquad[A_{\widebreve{FQ}_{n+1}}]^\mu_\nu=\delta_{\mu\nu}(n+1-4\lceil1/2\deg\mu\rceil).$$
Note that given $\mu\in\Z_2^{n+1}$ and $\nu\in\Z_2^n$, we can also express $\iota^\mu_\nu=\delta_{\mu\nu}=[\iota^\dag]^\nu_\mu$ if $\nu$ is even and $\iota^\mu_\nu=\im\delta_{\mu,\nu+\epsilon_{n+1}}$, so $[\iota^\dag]^\nu_\mu=-\im\delta_{\mu,\nu+\epsilon_{n+1}}$ if $\nu$ is odd. Now, we compute
$$[\iota^\dag A_{\breve Q_{n+1}}\iota]^\mu_\nu=
\begin{cases}
\delta_{\mu\nu}(n+1-2\deg\mu)&\text{if $\mu$ is even,}\\
\delta_{\mu\nu}(n-1-2\deg\mu)&\text{if $\mu$ is odd,}
\end{cases}
$$
which is what we wanted.
\end{proof}

\subsection{Anticommutative squared/halved hypercube}
Finally, let us have a look on the case of the squared hypercube $Q_n^2$ also known as the halved hypercube $\frac{1}{2}Q_{n+1}$. The situation here is actually slightly more complicated, so in order to avoid overloading this section with many abstract definitions, let us comment on this case rather briefly and informally.

The \emph{squared hypercube graph} $Q_n^2$ is the Cayley graph of $\Z_2^n$ with respect to the generating set $\{\epsilon_i\}_{i=1}^n\cup\{\epsilon_i+\epsilon_j\}_{i<j}$. That is, it was made by adding the diagonal edges to all 2-faces of $Q_n$. It is probably better known under the name \emph{halved hypercube graph} $\frac{1}{2}Q_{n+1}$ since it can also be obtained by \emph{halving} the $(n+1)$-dimensional hypercube -- taking one of the two components of the distance-2 graph of $Q_{n+1}$. Its spectrum is computed as
$$\lambda_\mu=\sum_{i=1}^n\tau_\mu(\epsilon_i)+\sum_{i<j}\tau_\mu(\epsilon_i+\epsilon_j)=l_d+\frac{1}{2}l_d^2-\frac{n}{2}=\frac{1}{2}\left((n+1-2d)^2-n-1\right),$$
where $l_d=n-2d$, $d=\deg\mu$.

\begin{defn}
The anticommutative deformation of the squared hypercube graph $Q_n^2$ in the sense of Def.~\ref{D.anticom} will be called the \emph{anticommutative squared hypercube graph} and denoted by $\breve Q_n^2:=\widebreve\Cay(\Z_2^n,\{\epsilon_i\}_{i=1}^n\cup\{\epsilon_i+\epsilon_j\}_{i<j})$.
\end{defn}

The quantum automorphism group of $Q_n^2=\frac{1}{2}Q_{n+1}$ is the anticommutative special orthogonal group $SO_{n+1}^{-1}$ \cite[Theorem~4.1]{GroAbSym}. This is known to be isomorphic to $PO_{n+1}^{-1}$ if $n$ is even. Hence, applying the 2-cocycle twist, we get $PO_{n+1}\simeq SO_{n+1}$ similarly as in the previous case.

\begin{prop}
The anticommutative squared hypercube $\breve Q_n^2$ is quantum isomorphic to the classical squared hypercube $Q_n^2$. For $n$ even, its quantum automorphism group is $SO_{n+1}\simeq PO_{n+1}$.
\end{prop}

It actually also holds that $\breve Q_n^2$ can be obtained by \emph{squaring} the hypercube graph $\breve Q_n$ in the following sense. Recall that given a graph $X$, its \emph{$k$-th power} is the graph defined on the same vertex set, having an edge between two vertices if and only if their distance in the original graph is less or equal to $k$. In particular, considering a Cayley graph $\Cay(\Gamma,S)$, we can construct its $k$-th power by considering the set $S\cup(S+S)\cup\dots\cup kS$. This idea directly translates into the anticommutative case or any other deformations.

If $n$ is even, then this graph can also be constructed by \emph{halving} the hypercube graph $\breve Q_{n+1}$. However, for $n$ odd, this \emph{does not} work. The argumentation goes as follows. Classically, taking a (connected) bipartite graph $X$, its distance-two graph splits into two connected components $X^+$ and $X^-$ -- the halved graphs. Often, as in the case of the hypercube, the two components are actually equal and denoted by $\frac{1}{2}X$. In our case, we can start with $\breve Q_{n+1}$ and denote by $A$ its adjacency matrix. The distance-two graph can be constructed simply by squaring the adjacency matrix. However, this way we create some loops and multiple edges, which one typically wants to remove. One can check that the quantum graph given by the adjacency matrix $A^2-(n+1)I$ already contains no loops and every edge has multiplicity two. That is, $\frac{1}{2}(A^2-(n+1)I)$ is a simple quantum graph. Exactly as in the classical case, where we need to remove $n+1$ loops at every vertex and then we are left with a double edge at each face diagonal. If $n$ is even, then this graph is defined on the quantum space $X$ with $C(X)=\Cl_{n+1}\simeq M_{2^{n/2}}(\C)\oplus M_{2^{n/2}}(\C)$ (see e.g.\ \cite{Tra06} Eqs.~(14--15) for the structure of Clifford algebras), so it is perfectly possible to construct a subgraph by restricting to one of the summands $M_{2^{n/2}}\subset X$. One can then check that this subgraph exactly equals to $\breve Q_n^2=\frac{1}{2}\breve Q_{n+1}$. In contrast, if $n$ is odd, then $\Cl_{n+1}\simeq M_{2^{(n+1)/2}}(\C)$ and hence the quantum space $X=M_{2^{(n+1)/2}}$ has no proper subspaces.

Let us now try to identify the core of the problem in terms of relations. This also reveals the reason why the quantum automorphism group of $\breve Q_n^2$ is not $SO_{n+1}$ for $n$ odd. (Actually, one can show that $SO_{n+1}$ is not at all monoidally equivalent to $SO_{n+1}^{-1}$ for $n$ odd.)

In the classical case, the vertex set of $Q_n^2=\frac{1}{2}Q_{n+1}$ is a subset of the vertex set of $Q_{n+1}$. In the language of algebras, there is a surjective $*$-homomorphism from $C(Q_{n+1})=\C\Z_2^{n+1}$ to $C(Q_n^2)=C(Q_n)=\C\Z_2^n$ mapping $\tau_i\mapsto\tau_i$ for $1\le i\le n$ and $\tau_{n+1}\mapsto\tau_1\cdots\tau_n$. Respecting this quotient structure, $SO_{n+1}^{-1}$ acts on $C(Q_n^2)$ by $\tau_i\mapsto\sum_{j=1}^{n+1}\tau_j\otimes u_i^j$, where $\tau_{n+1}=\tau_1\cdots\tau_n$ and $\Olg(SO_{n+1}^{-1})$ is a quotient of $\Olg(O_{n+1}^{-1})$ by a certain determinant-like condition $\sum_{\pi\in S_{n+1}}u^1_{\pi(1)}\cdots u^n_{\pi(n)}=1$. See \cite{GroAbSym} for details. 

The same story can be told also for the anticommutative version $\breve Q_n^2$ if $n$ even. If $n$ is odd, then this does not work since the element $\tau_1\cdots\tau_n\in C(\breve Q_n^2)=\Cl_n$ commutes with all the elements $\tau_1,\dots,\tau_n$ and hence cannot be the quotient image of $\tau_{n+1}$. Consequently, for $n$ odd, $\breve Q_n^2$ cannot be made by halving the anticommutative hypercube $\breve Q_{n+1}$. Instead, it is actually obtained by halving the \emph{$n$-times anticommutative and once commutative hypercube} $\breve Q_{(n\mid 1)}$, where $C(Q_{(n\mid 1)})=\Cl_n\oplus\C\Z_2$ generated by elements $\tau_1,\dots,\tau_{n+1}$ with $\tau_i^2=\tau_i=\tau_i^*$ such that $\tau_1,\dots,\tau_n$ mutually anticommute, but $\tau_{n+1}$ commutes with everything.

Following this logic, the quantum automorphism group of $\breve Q_n^2$ should be denoted by $SO_{(n\mid 1)}^{-1}$.

\section{Explicit examples of small sizes}
\label{sec.smallnoniso}

\subsection{Interpreting graphs over $M_2$}

In Section~\ref{secc.smallclass}, we proved that there are exactly four simple graphs over the quantum space $M_2$. They are uniquely determined by the \emph{number of quantum edges} $\rank\tilde A$, which can equal to 0, 1, 2, or 3. Instead of the number of quantum edges, which is not invariant under quantum isomorphisms, we should use the invariant number of edges $\eta^\dag A\eta\in\{0,4,8,12\}$. Recall that this number should be interpreted as the number of oriented edges.

We obtain the space $M_2$ by twisting the classical space of four points $X_4$ if we interpret it as the group $\Gamma=\Z_2\times\Z_2$. This is a special case of the situation from both Section~\ref{sec.rook} (where we studied $\Gamma=\Z_n\times\Z_n$) and Section~\ref{sec.cube} (where we studied $\Gamma=\Z_2^n$).

There are exactly four non-isomorphic simple Cayley graphs $\Cay(\Z_2\times\Z_2,S)$
\begin{itemize}
\item the empty graph for $S=\emptyset$,
\item two segments for $S=\{\epsilon_1\}$,
\item the square for $S=\{\epsilon_1,\epsilon_2\}$,
\item the full graph for $S=\{\epsilon_1,\epsilon_2,\epsilon_1+\epsilon_2\}=\Gamma\setminus\{0\}$.
\end{itemize}

Those graphs have exactly 0, 4, 8, and 12 directed edges respectively. Twisting them using the previously defined bicharacter, we get graphs over $M_2$. Namely, those must be exactly the only four simple quantum graphs on $M_2$.

\begin{prop}
\label{P.smallqiso}
All simple quantum graphs on $M_2$ are quantum Cayley graphs of $\Z_2\times\Z_2$. In particular, they are quantum isomorphic to classical graphs.
\end{prop}

The explicit form of the quantum isomorphism (in the form of a bigalois extension) was computed in \cite[Section~4.3]{Mat21}.



\subsection{Quantum graph, which is not quantum isomorphic to a classical graph}

In the last few sections, we constructed many graphs that were all quantum isomorphic to some classical ones. A natural question is then, whether there also is a graph, which is not quantum isomorphic to any classical graph. We have a simple criterion for finding these.

\begin{prop}
\label{P.nonqiso}
Let $(X,A)$ be a quantum graph. Denote by $u$ the fundamental representation of $\Aut^+(X,A)$. If the endomorphism algebra $\Mor(u,u)$ is non-commutative with respect to the Schur product, then $(X,A)$ is not quantum isomorphic to any classical graph. 
\end{prop}
\begin{proof}
Classically, the Schur product is just the pointwise multiplication of matrices, so it is commutative. Quantum isomorphisms preserve the Schur product since it is defined using the C*-algebra product $m$. In particular, quantum isomorphism must preserve the commutativity of the Schur product.
\end{proof}

Let us now bring some examples. First of all, any of the graphs containing a partial loop satisfies this property.

\begin{defn}
We say that a quantum graph $(X,\tilde A)$ contains a \emph{partial loop} if $A\bullet I\neq I\bullet A$.
\end{defn}

\begin{prop}
A quantum graph containing a partial loop is not quantum isomorphic to any classical graph.
\end{prop}
\begin{proof}
Follows directly from Proposition~\ref{P.nonqiso}
\end{proof}

In particular, all the graphs from Proposition~\ref{P.partialex} for $t\in(0,\pi/2)$ are not quantum isomorphic to any classical graph.

Now, we might be interested in an example of a {\em simple} quantum graph, which is not quantum isomorphic to a classical one. As follows from Proposition~\ref{P.smallqiso}, there is no example over $M_2$, so we need to go over $M_3$.

\begin{ex}
Consider the matrix
$$\lambda_8=\frac{1}{\sqrt2}
\begin{pmatrix}
1&0&0\\
0&1&0\\
0&0&-2
\end{pmatrix},
$$
which is one of the \emph{Gell-Mann matrices} spanning the Lie algebra $\su_3$. (We choose a different normalization than usual, so that $\Tr(\lambda_8^\dag\lambda_8)=3$.)

So, the adjacency matrix $A:=\lambda_8\otimes\lambda_8$ defines a simple graph on $M_3$. The corresponding endomorphism algebra $\Mor(u,u)$ contains $A$ by definition and $A^2$ by closedness under composition. A straightforward computation shows that $A\bullet A^2\neq A^2\bullet A$, so the graph $(M_3,\tilde A)$ is not quantum isomorphic to any classical graph.
\end{ex}

\bibliographystyle{halpha}
\bibliography{mybase}

\end{document}